\documentclass[journal]{new-aiaa}
\usepackage[utf8]{inputenc}
\usepackage{textcomp}

\usepackage{graphicx}
\usepackage{subcaption}
\usepackage{amsmath}

\usepackage{amsthm}
\newtheorem{proposition}{Proposition}
\usepackage[version=4]{mhchem}
\usepackage{siunitx}
\usepackage{longtable,tabularx}
\usepackage{array,ragged2e,booktabs}
\newcolumntype{L}{>{\RaggedRight\arraybackslash}X}
\usepackage{xcolor}

\title{Unified Relative Orbital Element Framework for Eccentric Formation Flying}

\author{Yuji Takubo\footnote{Ph.D. Candidate, Department of Aeronautics and Astronautics, 496 Lomita Mall. AIAA Student Member (ytakubo@stanford.edu).} and Simone D'Amico\footnote{Associate Professor, Department of Aeronautics and Astronautics, 496 Lomita Mall. AIAA Fellow.}}
\affil{Stanford University, Stanford, CA, 94305}

\begin{document}

\maketitle

\begin{abstract}
Quasi-nonsingular relative orbital elements (qnsROE) are among the most widely used state representations for formation flying and rendezvous in near-circular inclined orbits. They admit closed-form perturbed state transition matrices (STMs) and a direct geometric mapping, which underpins eccentricity/inclination (E/I)-vector separation for passive safety. Existing eccentric extensions of the qnsROE resolve this geometry in the radial--tangential--normal frame, where a residual semimajor-axis difference makes the E/I separation margin phase dependent. This paper introduces eccentric ROE that reduce to the qnsROE as \(e\to0\) and whose first-order geometry is resolved in the velocity--normal--binormal (VNB) frame, where the secular drift is confined to the velocity direction. Their first-order position and velocity maps yield a closed-form, phase-independent binormal--normal separation bound with an explicit margin for semimajor-axis residuals. Closed-form Keplerian and secular-\(J_2\) STMs and a decoupled impulsive control input matrix are derived for the same elements, and the elements are recovered as the physically normalized integration constants of the Cartesian equations of relative motion in the VNB frame. Together, these results provide a unified ROE framework for formation flying in general elliptic orbits.
\end{abstract}

\section{Introduction}
\label{sec:intro}

Distributed space systems (DSS) use controlled relative motion to realize apertures, baselines, and observation geometries unavailable to a monolithic spacecraft.
Many flown and planned formation-flying missions, including GRACE~\cite{montenbruck2006grace}, TanDEM-X~\cite{krieger2007tandemx}, PRISMA~\cite{damico2012prisma}, STARLING~\cite{kruger2023starfox}, SWARM-EX~\cite{agarwal2021swarmex}, and VISORS~\cite{koenig2023visors}, operate in near-circular inclined low Earth orbit, where relative motion admits a nearly time-invariant geometric description.
Other missions and concepts, including Magnetospheric Multiscale (MMS)~\cite{burch2016mms}, HelioSwarm~\cite{plice2019helioswarm}, ESCAPADE~\cite{lillis2020escapade}, PROBA-3~\cite{llorente2013proba3}, and the miniaturized distributed occulter/telescope (mDOT)~\cite{kolmas2016mdot}, require eccentric chief orbits.
The resulting variations in orbital radius, speed, and flight-path angle complicate relative-geometry design, passive-safety analysis, perturbation modeling, and control.

Formation designs in these missions rely on reduced-order models of the relative dynamics and their bounded periodic solutions.
The Hill--Clohessy--Wiltshire (HCW) equations~\cite{clohessy1960} describe linear relative motion about a circular chief in the radial--tangential--normal (RTN) frame, and the Tschauner--Hempel equations~\cite{tschauner1965} and the Yamanaka--Ankersen (YA) state transition matrix (STM)~\cite{yamanaka2002} extend this Cartesian description to eccentric Keplerian orbits.
Bounded and invariant formations follow from constraints on the initial conditions, such as energy matching~\cite{alfriend2010} and the $J_2$-invariance conditions of Schaub and Alfriend~\cite{schaub2001j2}.
Relative orbital elements (ROE) provide a complementary description that admits analytical perturbation theories through the variational equations ~\cite{brouwer1959,vallado2001fundamentals}.
Element differences at a reference epoch form a complete set of integration constants of the linearized dynamics, and first-order element-to-Cartesian maps make this correspondence explicit~\cite{casotto1993,schaub2004,gim2003,alfriend2010,sullivan2017survey}.
Because any two such sets are related by a constant nonsingular matrix, choosing an ROE set amounts to choosing a normalization of the Cartesian modes.

Quasi-nonsingular ROE (qnsROE)~\cite{damico2010thesis} are among the most widely used state representations for formation flying in near-circular inclined orbits.
To first order, they parameterize the HCW solution with relative eccentricity and inclination vectors that describe the relative orbit geometry directly.
They admit closed-form perturbed STMs~\cite{koenig2017stm,guffanti2019srp,gaias2020analytical} and underpin eccentricity/inclination (E/I)-vector separation~\cite{damico2006ei}, the standard passive-safety rule.
For an eccentric chief, however, the qnsROE position map acquires anomaly-dependent factors and is normalized by the osculating orbit radius~\cite{willis2023thesis}.
Recovering the geometric transparency of the near-circular case has therefore motivated several eccentric element sets.
Chernick and D'Amico decoupled in-plane and out-of-plane impulsive control~\cite{chernick2018,chernick2021}, and Sullivan and D'Amico recovered part of the geometric structure while retaining the radius normalization~\cite{sullivan2017nkf}.
Delurgio and D'Amico identified eccentric ROE (EROE) with the YA integration constants, normalized the map by the chief semimajor axis, and gave a quasi-nonsingular form that reduces to the qnsROE as $e\to0$~\cite{delurgio2024}.
This set extends the qnsROE--HCW correspondence to eccentric orbits, although it resolves the geometry in the RTN frame.

E/I-vector separation must tolerate the residual semimajor-axis difference left by navigation and control errors.
For a circular chief, this residual adds a constant radial offset and a secular along-track drift, and a fixed margin for the offset makes radial--normal (RN) separation insensitive to the drift~\cite{damico2006ei}.
For an eccentric chief, the drift acquires a radial component that grows linearly in time, so no fixed RN margin suffices.
Section~\ref{sec:rtn-limitation} shows that this phase dependence originates in the RTN projection rather than in the choice of elements.
To first order, the drift acts as a time offset along the chief trajectory and therefore displaces the deputy along the chief velocity.
The velocity--normal--binormal (VNB) frame confines the drift to the velocity axis, so binormal--normal (BN) separation requires only a fixed margin for the bounded binormal offset for all $0\le e<1$.
Albert and Schaub formulated VNB relative motion in classical-element differences without defining a relative element set~\cite{albert2023}.
Leveraging the dynamics in the VNB frame (or tangential frame (TAN)), Peters and Noomen generalized E/I-vector separation to elliptic chiefs with the C-set of relative elements~\cite{peters2021cotangential}.
They also mapped the C-set linearly to the YA integration constants and to the Cartesian state in the VNB frame.
The C-set, however, resolves the relative E/I vectors along the chief apsidal line and does not reduce to the qnsROE as $e\to0$.
As a result, the guidance architecture of Ref.~\cite{buratti2026fsm}, for example, converts between the two sets.

Table~\ref{tab:roe_representation_comparison} summarizes the resulting gap: the qnsROE lineage establishes phase-robust separation only near $e=0$, and the C-set, which establishes it for $0\le e<1$, does not reduce to the qnsROE.
This paper unifies these lines of work in a single ROE set by closing the correspondence between ROE and the integration constants of the linearized relative motion in the VNB frame.
The proposed EROE are defined from orbital-element differences, and their first-order position map in the VNB frame has the harmonic form of the qnsROE map.
They are recovered independently as the physically normalized integration constants of the Cartesian VNB equations of relative motion.
The C-set is a different normalization of the same modes, and the proposed set inherits its safe-orbit geometry.
The Delurgio--D'Amico EROE map to the proposed set by a regular transformation that reduces to the identity at $e=0$.
The proposed set thus plays for the VNB solution the role that the qnsROE play for the HCW solution and the Delurgio--D'Amico EROE for the YA solution.

The contributions of this paper are threefold.
First, through their first-order position map in the VNB frame, the proposed EROE yield a phase-robust E/I-vector separation bound in the BN plane for $0\le e<1$.
As $e\to0$, the EROE reduce to the qnsROE and the bound recovers near-circular E/I-vector separation.
Nonlinear simulations with a semimajor-axis residual show that formations sized by this bound remain outside the keep-out zone, whereas formations sized with the same margin by the RN-plane separation condition of Ref.~\cite{delurgio2024} can enter it.
Second, closed-form Keplerian and secular mean-$J_2$ STMs and an impulsive control input matrix are derived for the proposed EROE.
The input matrix decouples in-plane and out-of-plane control for all $0\le e<1$, and each result reduces to its qnsROE counterpart at $e=0$.
Third, the Cartesian VNB equations of relative motion are solved directly, yielding a closed-form fundamental matrix and its analytical inverse.
To the best of the authors' knowledge, no YA-type solution has previously been derived directly in the VNB frame.
Normalizing the modes of this solution by harmonic amplitude, physical time offset, and secular shear rate makes its integration constants the proposed EROE, so the EROE-to-Cartesian map and its inverse require no separate transformation.

The remainder of the paper is organized as follows.
Section~\ref{sec:rtn-limitation} shows that the loss of phase-robust geometry originates in the RTN projection and is removed by the VNB projection.
Section~\ref{sec:new-eroes} defines the proposed EROE, relates them to existing element sets, and develops their geometry in the VNB frame and their passive-safety conditions.
Section~\ref{sec:stm-control} derives the STMs and the control input matrix.
Section~\ref{sec:cartesian-derivation} derives the Cartesian VNB fundamental matrix directly and identifies its integration constants with the EROE.
Section~\ref{sec:numerical-results} evaluates the passive-safety benefit under a residual semimajor-axis difference and the finite-horizon accuracy of the STMs, and Sec.~\ref{sec:conclusions} concludes.

\begin{table*}[!tbp]
\centering
\caption{Comparison of representative quasi-nonsingular and eccentric relative orbital element sets.}
\label{tab:roe_representation_comparison}
\fontsize{8.5}{9.3}\selectfont
\setlength{\tabcolsep}{3pt}
\renewcommand{\arraystretch}{1.02}
\setlength{\abovedisplayskip}{1pt}
\setlength{\belowdisplayskip}{1pt}
\setlength{\abovedisplayshortskip}{1pt}
\setlength{\belowdisplayshortskip}{1pt}
\setlength{\jot}{0pt}
\begin{tabularx}{\textwidth}{@{}>{\RaggedRight\arraybackslash}p{0.11\textwidth} X >{\Centering\arraybackslash}p{0.07\textwidth} >{\Centering\arraybackslash}p{0.08\textwidth} >{\Centering\arraybackslash}p{0.09\textwidth} >{\Centering\arraybackslash}p{0.09\textwidth} >{\Centering\arraybackslash}p{0.10\textwidth}@{}}
\toprule
\textbf{Element set}
& \textbf{Definition (six elements)}
& \textbf{Regular at $e=0$}
& \textbf{Reduces to qnsROE as $e\to0$}
& \textbf{Geometric-map normalization}
& \textbf{In-/out-of-plane decoupled control}
& \textbf{Phase-robust E/I separation for $\delta a\ne0$} \\
\midrule

qnsROE~\cite{damico2010thesis}
&
$\begin{aligned}
\delta a&=\Delta a/a,\\
\delta\lambda&=\Delta u+\cos i\,\Delta\Omega,\\
\delta e_x&=\Delta e_x=e_d\cos\omega_d-e\cos\omega,\\
\delta e_y&=\Delta e_y=e_d\sin\omega_d-e\sin\omega,\\
\delta i_x&=\Delta i,\\
\delta i_y&=\sin i\,\Delta\Omega.
\end{aligned}$
& $\checkmark$
& identity
& $a$-normalized
& $e\approx0$
& $e\approx0$
\\
\midrule

Sullivan--D'Amico~\cite{sullivan2017nkf}
&
With $\xi=(1+e^2/2)/\eta^3$,\newline
$\begin{aligned}
\delta a&=\Delta a/a,\\
\delta\lambda&=\xi\, (\Delta u +\cos i\,\Delta\Omega)\\
&\quad +(1-\xi)\Bigl(\frac{-\sin\omega\,\Delta e_x+\cos\omega\,\Delta e_y}{e}\\
&\qquad +\cot i\, \sin i\,\Delta\Omega \Bigr),\\
\delta e_x&=\frac{\cos\omega\,\Delta e_x+\sin\omega\,\Delta e_y}{\eta^2},\\
\delta e_y&=\frac{1}{\eta^3}\bigl[-\sin\omega\,\Delta e_x+\cos\omega\,\Delta e_y\\
&\quad-e\, (\Delta u+\cos i\,\Delta\Omega)
+e\cot i\,\sin i\,\Delta\Omega\bigr],\\
\delta i_x&=\Delta i,\\
\delta i_y&=\sin i\,\Delta\Omega.
\end{aligned}$
& $\times$
& $\times$
& $r$-normalized
& $0<e<1$
& ---
\\
\midrule

Chernick--D'Amico~\cite{chernick2018}
&
$\begin{aligned}
\delta a&=\Delta a/a,\\
\delta\lambda&=\Delta M+\eta\bigl(\Delta\omega+\cos i\,\Delta\Omega\bigr),\\
\delta e_x&=\Delta e,\\
\delta e_y&=\Delta\omega+\cos i\,\Delta\Omega,\\
\delta i_x&=\Delta i,\\
\delta i_y&=\sin i\,\Delta\Omega.
\end{aligned}$
& $\times$
& $\times$
& $\times$
& $0<e<1$
& ---
\\
\midrule

Peters--Noomen\newline C-set~\cite{peters2021cotangential}
&
$\begin{aligned}
C_1&=\eta^2\Delta a-2ae\,\Delta e,\\
C_2&=\eta^2e\,\Delta a-a\chi\,\Delta e,\\
C_3&=-a\eta^2e\,(\Delta\omega+\cos i\,\Delta\Omega),\\
C_4&=a\bigl(\Delta\omega+\cos i\,\Delta\Omega+\eta^{-1}\Delta M\bigr),\\
C_5&=-a\eta^2(\cos\omega\,\Delta i+\sin\omega\sin i\,\Delta\Omega),\\
C_6&=a\eta^2(\sin\omega\,\Delta i-\cos\omega\sin i\,\Delta\Omega).
\end{aligned}$
& $\checkmark^{\mathrm{b}}$
& $\times$
& $a$-normalized$^{\mathrm{c}}$
& $0\le e<1$
& $0\le e<1$
\\
\midrule

Delurgio--D'Amico~\cite{delurgio2024}
&
$\begin{aligned}
\delta a&=\eta^2\Delta a/a,\\
\delta\lambda&=\frac{\Delta M}{\eta}
+\eta^2\bigl(\Delta\omega+\cos i\,\Delta\Omega\bigr),\\
\delta e_x&=\Delta e\cos\omega+\frac{e}{\eta}\Delta M\sin\omega,\\
\delta e_y&=\Delta e\sin\omega-\frac{e}{\eta}\Delta M\cos\omega,\\
\delta i_x&=\eta^2\Delta i,\\
\delta i_y&=\eta^2\sin i\,\Delta\Omega.
\end{aligned}$
& $\checkmark^{\mathrm{a}}$
& $\checkmark$
& $a$-normalized
& $0\le e<1$
& $e\approx0$
\\
\midrule

\textbf{This work}
&
With $E_\parallel=\chi\Delta e-\eta^2e\,\Delta a/a$ and $E_\perp=\eta^2e\,(\Delta\omega+\cos i\,\Delta\Omega)$,\newline
$\begin{aligned}
\delta a&=\Delta a/a,\\
\delta\lambda&=\Delta M+\frac{\eta^3}{\chi}
\bigl(\Delta\omega+\cos i\,\Delta\Omega\bigr),\\
\delta e_x&=E_\parallel\cos\omega-E_\perp\sin\omega,\\
\delta e_y&=E_\parallel\sin\omega+E_\perp\cos\omega,\\
\delta i_x&=\eta^2\Delta i,\\
\delta i_y&=\eta^2\sin i\,\Delta\Omega.
\end{aligned}$
& $\checkmark^{\mathrm{a}}$
& $\checkmark$
& $a$-normalized
& $0\le e<1$
& $0\le e<1$
\\
\bottomrule
\end{tabularx}

\smallskip
\begin{minipage}{\textwidth}
\footnotesize
Notes: $\eta=\sqrt{1-e^2}$ and $\chi=1+e^2$, and $\Delta$ denotes a deputy-minus-chief element difference.
$^{\mathrm{a}}$Regular in the quasi-nonsingular form (Ref.~\cite{delurgio2024} and Appendix~\ref{app:regularization}).
$^{\mathrm{b}}$Through the maps to Cartesian coordinates and to the YA constants~\cite{peters2021cotangential}.
$^{\mathrm{c}}$After division by $a$, the N and B rows of the position map share the scaling of this work, whereas the V row carries an additional $C_3$ term (cf. Sec.~\ref{sec:equivalence}).
\end{minipage}
\end{table*}

\section{Frame Dependence of Phase-Robust Geometry}
\label{sec:rtn-limitation}

This section shows, using the existing YA solution, that for an eccentric chief the loss of phase robustness originates in the RTN projection rather than in the choice of relative elements.

Let $\boldsymbol r_c$ and $\boldsymbol v_c$ denote the chief inertial position and velocity.
The chief RTN and VNB frames are defined by
\begin{equation}
\hat R=\frac{\boldsymbol r_c}{\|\boldsymbol r_c\|},\qquad
\hat N=\frac{\boldsymbol r_c\times\boldsymbol v_c}{\|\boldsymbol r_c\times\boldsymbol v_c\|},\qquad
\hat T=\hat N\times\hat R,\qquad
\hat V=\frac{\boldsymbol v_c}{\|\boldsymbol v_c\|},\qquad
\hat B=\hat V\times\hat N.
\label{eq:frames}
\end{equation}
The axis $\hat N$ is shared by the two frames, and $\hat B$ lies in the orbital plane, perpendicular to the chief velocity.
Stacking the position components and their derivatives in each rotating frame gives the relative states
\begin{equation}
\boldsymbol x_{\mathrm{RTN}}=\begin{bmatrix}\rho_R&\rho_T&\rho_N&\dot\rho_R&\dot\rho_T&\dot\rho_N\end{bmatrix}^{\top},
\qquad
\boldsymbol x_{\mathrm{VNB}}=\begin{bmatrix}\rho_V&\rho_N&\rho_B&\dot\rho_V&\dot\rho_N&\dot\rho_B\end{bmatrix}^{\top}.
\label{eq:relative-states}
\end{equation}
The TAN frame of Ref.~\cite{peters2021cotangential} has axes $\hat V$, $-\hat N$, and $-\hat B$.
The chief classical elements are the semimajor axis $a$, eccentricity $e$, inclination $i$, right ascension of the ascending node $\Omega$, argument of periapsis $\omega$, and mean anomaly $M$.
The chief true anomaly is denoted by $\nu$, the gravitational parameter by $\mu$, and the deputy-minus-chief difference operator by $\Delta$.
The chief scalars used throughout the paper are
\begin{equation}
\begin{gathered}
n=\sqrt{\mu/a^3},\qquad
\eta=\sqrt{1-e^2},\qquad
\chi=1+e^2=2-\eta^2,\\
u=M+\omega,\qquad
\theta=\omega+\nu,\qquad
e_x=e\cos\omega,\qquad
e_y=e\sin\omega,\\
k=1+e\cos\nu=1+e_x\cos\theta+e_y\sin\theta,\qquad
S=e\sin\nu=e_x\sin\theta-e_y\cos\theta,\\
r_c=\|\boldsymbol r_c\|=\frac{a\eta^2}{k},\qquad
v_c=\|\boldsymbol v_c\|,\\
q=\sqrt{k^2+S^2}=\sqrt{2k-\eta^2}=\frac{\eta v_c}{na},\\
\dot\theta=\frac{nk^2}{\eta^3},\qquad
\omega_V=\frac{k}{q^2}\dot\theta=\frac{nk^3}{\eta^3q^2},\qquad
\tau=t-t_0,
\end{gathered}
\label{eq:chief-notation}
\end{equation}
where $u$ is the mean argument of latitude, $\theta$ is the true argument of latitude, and $\boldsymbol e=[e_x,e_y]^{\top}$ is the chief eccentricity vector.
The chief quasi-nonsingular (QNS) orbital elements are thus $\boldsymbol\alpha_c^{(\mathrm{qns})}=[a,u,e_x,e_y,i,\Omega]^{\top}$.
The qnsROE and the Delurgio--D'Amico EROE of Table~\ref{tab:roe_representation_comparison} are denoted $\delta\boldsymbol\alpha^{(\mathrm{qns})} = [\delta a^{(\mathrm{qns})}, \delta\lambda^{(\mathrm{qns})}, \delta e_x^{(\mathrm{qns})}, \delta e_y^{(\mathrm{qns})}, \delta i_x^{(\mathrm{qns})}, \delta i_y^{(\mathrm{qns})}]^{\top}$ and $\delta\boldsymbol\alpha^{(D)} = [\delta a^{(D)}, \delta\lambda^{(D)}, \delta e_x^{(D)}, \delta e_y^{(D)}, \delta i_x^{(D)}, \delta i_y^{(D)}]^{\top}$, and symbols without a superscript denote the proposed EROE of Sec.~\ref{sec:new-eroes}.
In every set, the relative eccentricity and inclination vectors are $\delta\boldsymbol e=[\delta e_x,\delta e_y]^{\top}$ and $\delta\boldsymbol i=[\delta i_x,\delta i_y]^{\top}$, carrying the superscript of the set.
The parameter $k$ fixes the chief radius $r_c$, and $S$ and $q$ carry the chief flight-path geometry, with $k/q = \cos \gamma$ and $S/q = \sin\gamma$ where $\gamma$ is the chief flight-path angle.
Both frames rotate about $\hat N$, the RTN frame at $\dot\theta$ and the VNB frame at $\omega_V=\dot\theta-\dot\gamma$, and the inertial relative velocity follows from Eq.~\eqref{eq:relative-states} by adding $\boldsymbol\omega\times\boldsymbol\rho$, with $\boldsymbol\omega=\dot\theta\hat N$ or $\omega_V\hat N$.

The YA solution expresses first-order Cartesian relative motion about an eccentric Keplerian chief as $\boldsymbol x_{\mathrm{RTN}}(t)=\boldsymbol F_{\mathrm{RTN}}(t,t_0)\boldsymbol c_{\mathrm{RTN}}$, where $\boldsymbol F_{\mathrm{RTN}}$ is a fundamental matrix and $\boldsymbol c_{\mathrm{RTN}}$ contains six integration constants~\cite{yamanaka2002,guffanti2022optimal}.
Under the normalization, these constants can be identified with the Delurgio--D'Amico EROE at a reference epoch, $\boldsymbol c_{\mathrm{RTN}}=\delta\boldsymbol\alpha^{(D)}(t_0)$~\cite{delurgio2024}, so the columns of $\boldsymbol F_{\mathrm{RTN}}$ are the Cartesian responses of the individual elements.
Separating the bounded from the time-growing contributions of the position block gives~\cite{guffanti2022optimal}
\begin{equation}
\frac{\boldsymbol{\rho}_{\mathrm{RTN}}(t)}{a}
=
\begin{bmatrix}
\dfrac{\delta a^{(D)}}{k}
-\delta e_x^{(D)}\cos\theta-\delta e_y^{(D)}\sin\theta\\[8pt]
\dfrac{\delta\lambda^{(D)}}{k}
+\left(\dfrac{1}{k}+1\right)
\left(\delta e_x^{(D)}\sin\theta-\delta e_y^{(D)}\cos\theta\right)\\[8pt]
\dfrac{1}{k}\left(\delta i_x^{(D)}\sin\theta
-\delta i_y^{(D)}\cos\theta\right)
\end{bmatrix}
\underbrace{{}-
\frac{3n\tau}{2\eta^3}
\begin{bmatrix}
e\sin\nu\\
k\\
0
\end{bmatrix}\delta a^{(D)}}_{\boldsymbol\rho_{\mathrm{RTN},\mathrm{sec}}(t)/a},
\label{eq:rtn-response}
\end{equation}
where every Delurgio--D'Amico EROE on the right-hand side is evaluated at $t_0$, and $\boldsymbol\rho_{\mathrm{RTN},\mathrm{sec}}$ is the secular phase contribution.
For a circular chief, a semimajor-axis residual gives $\Delta\rho_R=\Delta a_{\mathrm{err}}$ and $\Delta\rho_{T,\mathrm{sec}}=-\tfrac32n\tau\,\Delta a_{\mathrm{err}}$.
The RN clearance must therefore allow for the constant radial offset even though the secular contribution is purely along-track.
For $e>0$, the secular radial component is nonzero away from the apsides and its amplitude grows linearly with time.
An E/I-vector separation condition formulated in the RN projection~\cite{delurgio2024} does not, in general, provide a time-invariant lower bound when $\delta a^{(D)}\ne0$.

Rotating the same solution into the VNB frame changes this conclusion.
With $S$ and $q$ of Eq.~\eqref{eq:chief-notation}, the rotation from RTN to VNB components is
\begin{equation}
\boldsymbol{\rho}_{\mathrm{VNB}}
=
\boldsymbol{R}_{\mathrm{VR}}\,\boldsymbol{\rho}_{\mathrm{RTN}},
\qquad
\boldsymbol{R}_{\mathrm{VR}}
=
\frac{1}{q}
\begin{bmatrix}
S & k & 0\\
0 & 0 & q\\
k & -S & 0
\end{bmatrix}.
\label{eq:rtn-to-vnb}
\end{equation}
The rows of $\boldsymbol R_{\mathrm{VR}}$ are $\hat V$, $\hat N$, and $\hat B$ resolved in RTN, and the secular contribution of Eq.~\eqref{eq:rtn-response} becomes
\begin{equation}
\boldsymbol\rho_{\mathrm{VNB},\mathrm{sec}}
=\boldsymbol R_{\mathrm{VR}}\,\boldsymbol\rho_{\mathrm{RTN},\mathrm{sec}}
=-\frac{3n\tau}{2\eta^3}\,aq\,\delta a^{(D)}
\begin{bmatrix}1&0&0\end{bmatrix}^{\top}.
\label{eq:vnb-secular}
\end{equation}
The accumulated phase enters neither $B$ nor $N$, so the BN projection remains bounded and periodic under Keplerian dynamics even when $\delta a^{(D)}\ne0$.
The bounded semimajor-axis contribution still shifts the binormal geometry and must be included in the separation margin.
Peters and Noomen observed the same independence in C-set coordinates~\cite{peters2021cotangential}, and Eq.~\eqref{eq:vnb-secular} shows that it follows from the YA solution by a rotation alone, independently of the element set.

\section{Proposed Eccentric Relative Orbital Elements}
\label{sec:new-eroes}

This section defines the proposed EROE, derives their first-order position and velocity maps in the VNB frame, relates them to existing element sets, and develops the resulting passive-safety conditions in the BN plane.

\subsection{Definition}
\label{sec:definition}

The proposed EROE are defined as
\begin{equation}
\delta\boldsymbol\alpha=
\begin{bmatrix}
\delta a \\
\delta\lambda \\
\delta e_x \\
\delta e_y \\
\delta i_x \\
\delta i_y
\end{bmatrix}
 =
\begin{bmatrix}
\dfrac{\Delta a}{a} \\
\Delta M+\dfrac{\eta^3}{\chi} (\Delta\omega+\cos i\,\Delta\Omega) \\
E_\parallel\cos\omega-E_\perp\sin\omega \\
E_\parallel\sin\omega+E_\perp\cos\omega\\
\eta^2\Delta i \\
\eta^2\sin i\,\Delta\Omega
\end{bmatrix}.
\label{eq:roe-definition}
\end{equation}
The in-plane shape components are
\begin{equation}
E_\parallel=\chi\Delta e-\eta^2e\,\dfrac{\Delta a}{a}, \qquad
E_\perp=\eta^2e\,(\Delta\omega+\cos i\,\Delta\Omega).
\label{eq:classical-definitions}
\end{equation}
Relative to the qnsROE, the definition retains $\delta a$, scales the inclination pair by $\eta^2$, and reweights the periapsis contribution to the longitude by $\eta^3/\chi$.
It replaces the eccentricity-vector difference with $(E_\parallel,E_\perp)$, rotated from the periapsis direction into the axes of $\boldsymbol e$, and the parallel component includes an energy correction proportional to $\Delta a$.
Every weight reduces to unity, and the energy correction vanishes, as $e\to0$.
Section~\ref{sec:vnb-geometry} derives these weights from the first-order VNB position map in classical element differences, and Sec.~\ref{sec:cartesian-derivation} shows independently that they coincide with a physical normalization of the integration constants of the Cartesian VNB equations of motion.
Other normalizations, such as that of the C-set, are related to Eq.~\eqref{eq:roe-definition} by the constant maps of Sec.~\ref{sec:equivalence}.

The classical-element form of Eq.~\eqref{eq:roe-definition} is singular at $e=0$, where $\omega$ is undefined.
Appendix~\ref{app:regularization} rewrites $\delta\lambda$ and $\delta\boldsymbol e$ in node-corrected eccentricity differences, giving a forward and inverse transformation that is regular for $0\le e<1$.

\subsection{Geometry in the VNB Frame}
\label{sec:vnb-geometry}

This subsection derives the VNB position and velocity maps of the proposed EROE, Eqs.~\eqref{eq:position} and~\eqref{eq:velocity}.
First, the first-order VNB position in classical element differences is written as~\cite{albert2023}
\begin{equation}
\begin{aligned}
\frac{\rho_V}{a}&=\frac q\eta\Delta M+\frac{\eta^2S}{kq}\frac{\Delta a}{a}
+\frac{2\sin\nu}{q}\Delta e+\frac{\eta^2}{q}(\Delta\omega+\cos i\,\Delta\Omega),\\
\frac{\rho_N}{a}&=\frac{\eta^2}{k}(\sin\theta\,\Delta i-\cos\theta\sin i\,\Delta\Omega), \\
\frac{\rho_B}{a}&=\frac{\eta^2}{q}\frac{\Delta a}{a}
-\frac{2e+\chi\cos\nu}{kq}\Delta e-\frac{\eta^2S}{kq}(\Delta\omega+\cos i\,\Delta\Omega).
\end{aligned}
\label{eq:vnb-classical-map}
\end{equation}
Each component is split into a center term and a scaled first harmonic of the anomaly, and the harmonic coefficients define the elements.
Scaling the binormal component by $kq$ gives the decomposition into a constant offset and a harmonic as
\begin{equation}
\frac{kq\,\rho_B}{a}=\frac{\Delta p}{a}-E_\parallel\cos\nu-E_\perp\sin\nu,
\qquad
\delta p := \frac{\Delta p}{a}=\eta^2\frac{\Delta a}{a}-2e\,\Delta e
=\frac{\eta^4\delta a-2(e_x\delta e_x+e_y\delta e_y)}{\chi},
\label{eq:semilatus-offset}
\end{equation}
where $p=a(1-e^2)$.
The last form follows from Eq.~\eqref{eq:roe-definition}, and the harmonic coefficients are the shape components of Eq.~\eqref{eq:classical-definitions}.
With the true argument of latitude $\theta=\omega+\nu$ and the EROE of Eq.~\eqref{eq:roe-definition}, Eq.~\eqref{eq:vnb-classical-map} becomes
\begin{equation}
\begin{aligned}
\frac{\rho_V}{a}&=\frac q\eta\delta\lambda+\mathcal C\,\delta a
+\frac{2}{\chi q}(\delta e_x\sin\theta-\delta e_y\cos\theta)
&&=\frac q\eta\delta\lambda+\mathcal C\,\delta a+\frac{2\,\delta e}{\chi q}\sin(\theta-\phi_e),\\
\frac{\rho_N}{a}&=\frac{1}{k} (\delta i_x\sin\theta-\delta i_y\cos\theta)
&&=\frac{\delta i}{k}\sin(\theta-\phi_i),\\
\frac{\rho_B}{a}&=\frac{1}{kq} ( \delta p-\delta e_x\cos\theta-\delta e_y\sin\theta )
&&=\frac{1}{kq} (\delta p-\delta e\cos(\theta-\phi_e)),\\
\mathcal C&=\frac{\eta^2S}{q}\left(\frac1k+\frac{2}{\chi}\right),
\end{aligned}
\label{eq:position}
\end{equation}
where $\delta e=\|\delta\boldsymbol e\|$, $\delta i=\|\delta\boldsymbol i\|$, $\phi_e=\operatorname{atan2}(\delta e_y,\delta e_x)$, and $\phi_i=\operatorname{atan2}(\delta i_y,\delta i_x)$.

Differentiating Eq.~\eqref{eq:position} along Keplerian motion, with $\dot\theta$ of Eq.~\eqref{eq:chief-notation}, $dk/d\theta=-S$, $dS/d\theta=k-1$, $dq/d\theta=-S/q$, and $\dot{\delta\lambda}=-\tfrac32n\delta a$ (cf. Sec.~\ref{sec:roe-stm}), and eliminating $\delta\lambda$ and $\delta p$ through the position components gives the VNB-frame velocity mapping as
\begin{equation}
\begin{aligned}
\frac{\dot\rho_V}{a\dot\theta}&=-\frac{S}{q^2}\frac{\rho_V}{a}-\frac{2k}{q^2}\frac{\rho_B}{a}+\frac{\eta^4}{2k^2q}\delta a,\\
\frac{\dot\rho_N}{a\dot\theta}&=\frac{S}{k}\frac{\rho_N}{a}+\frac{1}{k}(\delta i_x\cos\theta+\delta i_y\sin\theta),\\
\frac{\dot\rho_B}{a\dot\theta}&=S\left(\frac1k+\frac1{q^2}\right)\frac{\rho_B}{a}+\frac{1}{kq}(\delta e_x\sin\theta-\delta e_y\cos\theta).
\end{aligned}
\label{eq:velocity}
\end{equation}
Together, Eqs.~\eqref{eq:position} and~\eqref{eq:velocity} are the componentwise form of the instantaneous six-state map $\boldsymbol x_{\mathrm{VNB}}(t)=\boldsymbol\Psi_{\mathrm{VNB}}(t)\,\delta\boldsymbol\alpha(t)$, whose matrix is recovered independently in Sec.~\ref{sec:cartesian-identification}.

Based on Eq.~\eqref{eq:position}, Fig.~\ref{fig:VNB_projection} illustrates the resulting geometry of the osculating ellipse in the VNB frame.
The BN geometry is displaced from the chief by the relative semilatus rectum $\delta p$ of Eq.~\eqref{eq:semilatus-offset}, which reduces to $\delta a$, the radial offset of near-circular relative motion, as $e\to0$.
The osculating ratio of the V and B shape amplitudes is $2k/\chi$, which reduces to the familiar 2:1 ellipse as $e\to0$ and $\delta a\to0$.
However, because $k$ and $q$ vary with anomaly, the projected VB and BN loci are not ellipses for $e>0$.

\begin{figure}[hbt!]
\centering
\includegraphics[width=0.8\textwidth]{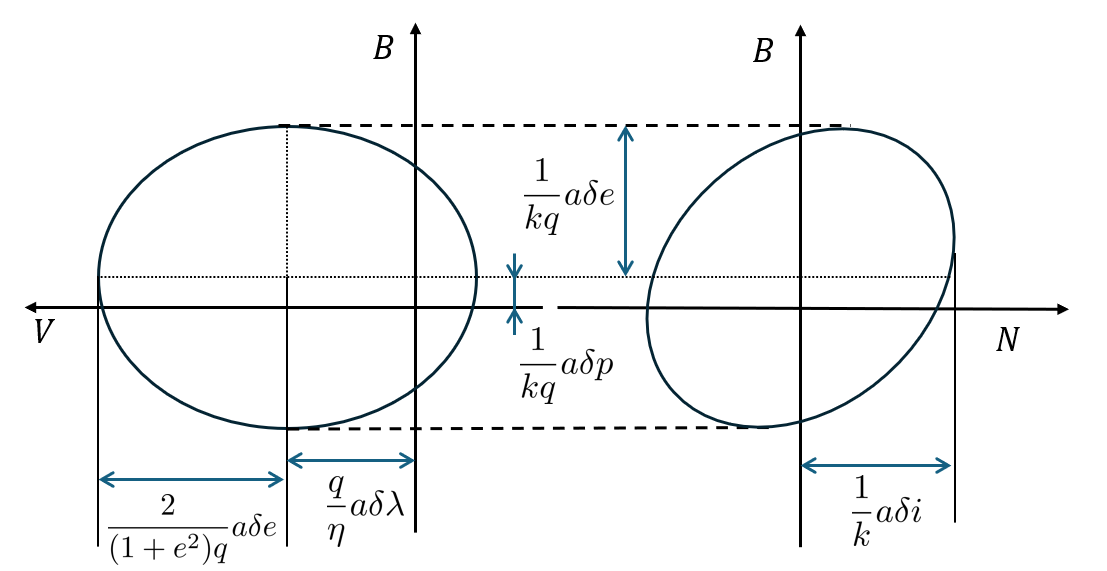}
\caption{Osculating relative motion geometry in the chief VNB frame for $\delta a=0$, projected onto the VB and BN planes.}
\label{fig:VNB_projection}
\end{figure}

As an example, Fig.~\ref{fig:frame-drift} demonstrates this effect using a deputy that differs from the chief of Sec.~\ref{sec:passive-safety-case-study} only by a \SI{100}{\metre} semimajor-axis offset at $e=0.4$.
The chief and deputy are propagated with nonlinear two-body dynamics for three chief orbits.
Figures~\ref{fig:frame-drift-rtn} and~\ref{fig:frame-drift-vnb} resolve the same trajectory in the two frames, so any difference between them is due solely to the projection.
The purple curves are the first-order phase/energy centers, that is, the nonharmonic terms of Eq.~\eqref{eq:rtn-response} in RTN frame and of Eq.~\eqref{eq:position} in VNB frame.
The nonlinear trajectories agree with the corresponding first-order maps to within \SI{1}{\metre}.
For a pure semimajor-axis offset, every Delurgio--D'Amico EROE except $\delta a^{(D)}$ vanishes, so the RTN trajectory coincides with its center, and the growing radial loops in Fig.~\ref{fig:frame-drift-rtn} are the secular radial term of Eq.~\eqref{eq:rtn-response} itself.
In the VNB-frame projection of Fig.~\ref{fig:frame-drift-vnb}, the same drift moves the center only along $V$.
The binormal offset $a\delta p/(kq)$ of the center stays bounded, and the binormal coordinate of the trajectory remains between \SI{59}{\metre} and \SI{140}{\metre}, as given by the offset $\eta^2\Delta a/q$ derived in Sec.~\ref{sec:passive-safety}.
Figure~\ref{fig:frame-drift-centered} repeats the VNB case with matched semilatus recta and is discussed there.

\begin{figure}[hbt!]
\centering
\begin{subfigure}[t]{0.32\textwidth}
\centering
\includegraphics[width=\linewidth]{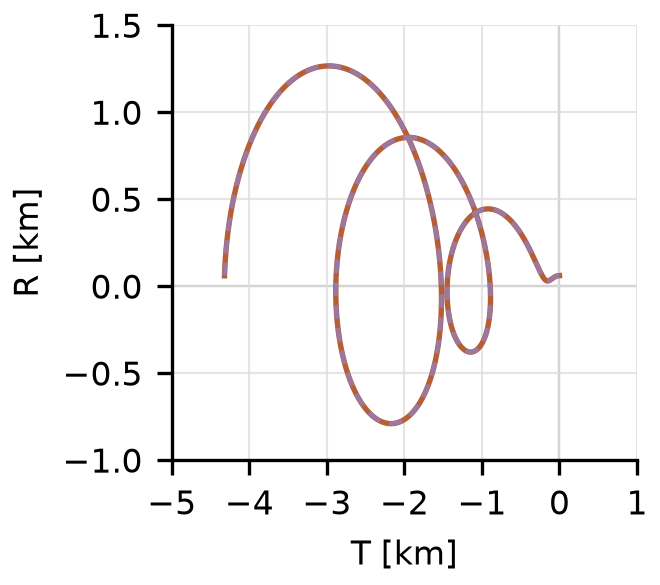}
\caption{RTN, pure $\Delta a$.}
\label{fig:frame-drift-rtn}
\end{subfigure}\hfill
\begin{subfigure}[t]{0.32\textwidth}
\centering
\includegraphics[width=\linewidth]{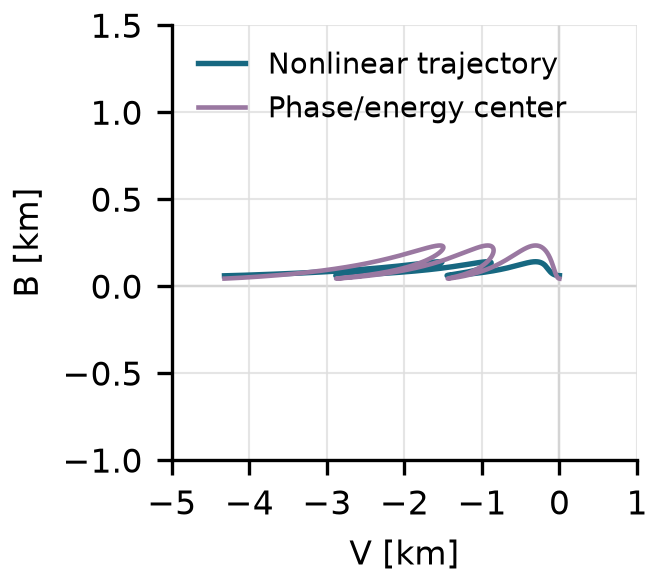}
\caption{VNB, pure $\Delta a$ ($\delta p\ne0$).}
\label{fig:frame-drift-vnb}
\end{subfigure}\hfill
\begin{subfigure}[t]{0.32\textwidth}
\centering
\includegraphics[width=\linewidth]{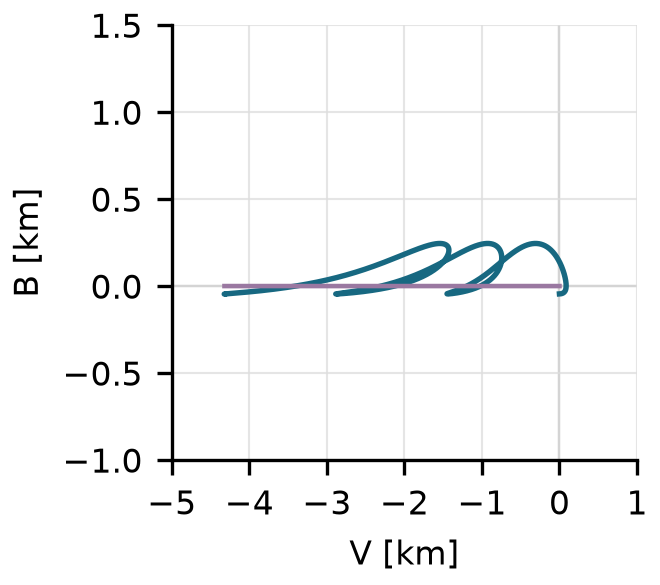}
\caption{VNB, $\delta p=0$.}
\label{fig:frame-drift-centered}
\end{subfigure}
\caption{Relative motion under a \SI{100}{\metre} semimajor-axis offset at $e=0.4$: (a) RTN projection of the pure-$\Delta a$ trajectory, (b) the same trajectory in VNB, and (c) the VNB trajectory after matching the semilatus recta. Purple curves denote the corresponding first-order phase/energy centers.}
\label{fig:frame-drift}
\end{figure}

\subsection{Relation to Existing Element Sets}
\label{sec:equivalence}

The proposed elements are related to the Delurgio--D'Amico EROE, and through them to the YA integration constants, and to the C-set by maps that are constant under Keplerian motion.
Let $\boldsymbol e^{\perp}=[-e_y,e_x]^{\top}$, and $\boldsymbol R(\cdot)$ the $2\times2$ rotation matrix.
\begin{proposition}[Relation to the Delurgio--D'Amico EROE and the C-set]
\label{prop:equivalence}
At every epoch of Keplerian relative motion about a non-equatorial chief, the following hold.
\begin{enumerate}
\item For $0\le e<1$, the proposed EROE and the Delurgio--D'Amico EROE of Table~\ref{tab:roe_representation_comparison} are related by
\begin{equation}
\begin{aligned}
\delta a&=\frac{\delta a^{(D)}}{\eta^{2}},\qquad
\delta\lambda=\frac{\eta}{\chi}\left(\delta\lambda^{(D)}-(\boldsymbol e^{\perp})^{\top}\delta\boldsymbol e^{(D)}\right),\qquad
\delta\boldsymbol i=\delta\boldsymbol i^{(D)},\\
\delta\boldsymbol e&=\left(\boldsymbol I_2+\boldsymbol e\boldsymbol e^{\top}\right)\delta\boldsymbol e^{(D)}
-\delta a^{(D)}\,\boldsymbol e+\delta\lambda^{(D)}\,\boldsymbol e^{\perp}.
\end{aligned}
\label{eq:map-delurgio}
\end{equation}
The map depends only on the chief eccentricity vector, is smooth for $0\le e<1$, and reduces to the identity at $e=0$.
\item For $0<e<1$, the proposed EROE and the C-set of Table~\ref{tab:roe_representation_comparison} are related by
\begin{equation}
\begin{gathered}
a\,\delta p=C_1,\qquad
a\eta^4\delta a=\chi C_1-2eC_2,\qquad
a\,\delta\lambda=\eta C_4+\frac{2e}{\eta\chi}C_3,\\
a\,\delta\boldsymbol e=-\boldsymbol R(\omega)
\begin{bmatrix}C_2\\ C_3\end{bmatrix},\qquad
a\,\delta\boldsymbol i=-\boldsymbol R(\omega)
\begin{bmatrix}C_5\\ C_6\end{bmatrix}.
\end{gathered}
\label{eq:map-cset}
\end{equation}
The map depends on the chief argument of periapsis and therefore has no limit as $e\to0$.
\end{enumerate}
\end{proposition}
\begin{proof}
Direct substitution of the definitions of Table~\ref{tab:roe_representation_comparison} into Eqs.~\eqref{eq:roe-definition}, \eqref{eq:classical-definitions}, and~\eqref{eq:semilatus-offset}, using $\boldsymbol R(-\omega)\,\delta\boldsymbol e^{(D)}=[\Delta e,\,-e\Delta M/\eta]^{\top}$, $C_2=-aE_\parallel$, $C_3=-aE_\perp$, and $[C_5,C_6]^{\top}=-a\boldsymbol R(-\omega)\,\delta\boldsymbol i$.
\end{proof}

Statement 1 ties the proposed set to the qnsROE lineage, developed in the RTN frame.
With the identification $\boldsymbol c_{\mathrm{RTN}}=\delta\boldsymbol\alpha^{(D)}(t_0)$ of Sec.~\ref{sec:rtn-limitation}, it makes $\delta\boldsymbol\alpha(t_0)$ a constant linear transformation of the YA integration constants.
Section~\ref{sec:cartesian-derivation} further shows that the proposed EROE are equivalent to the normalized integration constants of the Cartesian VNB equations of motion.
Equation~\eqref{eq:map-delurgio} reduces to the identity at $e=0$, where the Delurgio--D'Amico EROE and the proposed EROE both reduce to the qnsROE.
Statement 2 shows that the proposed set carries the C-set geometry within the VNB frame.
The N and B rows of Eq.~\eqref{eq:position} are equivalent to Eqs.~(52) and~(25) of Ref.~\cite{peters2021cotangential}.
The main difference between the two sets lies in the reference axes, where the C-set resolves the relative eccentricity and inclination vectors along the chief apsidal line, which is undefined at $e=0$, whereas Eq.~\eqref{eq:roe-definition} resolves them along the line of nodes, as the qnsROE do.
This choice gives the proposed set its qnsROE limit, so that the STMs, bounds, and input matrix derived below reduce continuously to their qnsROE counterparts as $e\to0$.
Furthermore, expressing $\rho_V$ in terms of $C_4$ introduces the additional term $2eqC_3/(\eta^2\chi)$ in Eq.~\eqref{eq:map-cset}, which lies outside the harmonic pair. In contrast, $\delta\lambda$ absorbs this contribution, leaving the harmonic terms associated with $\delta\boldsymbol e$ in $\rho_V$ and $\rho_B$, as shown in Eq.~\eqref{eq:position}.

\subsection{Passive Safety in the BN Plane}
\label{sec:passive-safety}

Equation~\eqref{eq:position} separates the two planes relevant to passive safety against the unstable mode.
The VB plane carries the Keplerian drift.
Its curve is closed for $\delta a=0$, whereas $\delta a\ne0$ adds secular motion along $V$.
The BN plane carries the bounded geometry, and its components do not contain $\delta\lambda$.
The minimum BN-plane separation from the chief is
\begin{equation}
d_{\mathrm{BN}}=a\min_{0\le\theta<2\pi}
\sqrt{\frac{(\delta p-\delta e_x\cos\theta-\delta e_y\sin\theta)^2}{k^2q^2}
+\frac{(\delta i_x\sin\theta-\delta i_y\cos\theta)^2}{k^2}}.
\label{eq:bn-minimum}
\end{equation}
This subsection first characterizes when the BN curve passes through the chief and how $\delta p$ sets its offset, and then bounds $d_{\mathrm{BN}}$ for centered designs.

\subsubsection{Crossing Conditions and Centering by the Relative Semilatus Rectum}

By Eq.~\eqref{eq:semilatus-offset}, the BN curve is offset from the chief by the relative semilatus rectum, which in the C-set is the relative altitude $a\delta p=C_1$ of the safe-orbit center~\cite{peters2021cotangential}.
For $\delta i>0$, $\rho_N$ vanishes only at $\theta\in\{\phi_i,\phi_i+\pi\}$, where $\rho_B=0$ requires $\delta p=\pm\delta e\cos(\phi_i-\phi_e)$.
The BN curve therefore passes through the chief when
\begin{equation}
\delta p^2\delta i^2=(\delta e_x\delta i_x+\delta e_y\delta i_y)^2.
\label{eq:bn-crossing}
\end{equation}
For $\delta i=0$, it does so when $|\delta p|\le\delta e$.
Equation~\eqref{eq:bn-crossing} is the boundary of the winding condition $|C_1|<|C_m\cos\tau_0|$ of Ref.~\cite{peters2021cotangential}, with $C_m=a\delta e$ and $\tau_0=\phi_i-\phi_e$ by Proposition~\ref{prop:equivalence}.
A crossing can be avoided either by exploiting the offset or by removing it.

An offset larger than the shape amplitude gives separation in the BV plane.
The inequalities $|\delta p-\delta e\cos(\theta-\phi_e)|\ge|\delta p|-\delta e$ and $kq\le(1+e)^2$ imply
\begin{equation}
|\delta p|>\delta e
\quad\Longrightarrow\quad
d_{\mathrm{BN}}\ge\frac{a(|\delta p|-\delta e)}{(1+e)^2}>0.
\label{eq:orientation-independent-bound}
\end{equation}
This option is sharply limited by the common origin of $\delta p$ and $\delta\boldsymbol e$.
With $\chi\delta p=\eta^4\delta a-2eE_\parallel$ and $\delta e^2=E_\parallel^2+E_\perp^2$, the quantity $\chi^2(\delta p^2-\delta e^2)$ is a concave quadratic in $E_\parallel$ with leading coefficient $4e^2-\chi^2=-\eta^4$.
Its maximum, attained at $E_\parallel=-2e\,\delta a$, is $\chi^2(\eta^4\delta a^2-E_\perp^2)$ because $\eta^4+4e^2=\chi^2$.
At fixed $E_\perp$, some $E_\parallel$ therefore satisfies $|\delta p|>\delta e$ if and only if
\begin{equation}
|E_\perp|<\eta^2|\delta a|.
\label{eq:orientation-independent-attainability}
\end{equation}
Moreover, $\chi|\delta p|\le\eta^4|\delta a|+2e\,\delta e$ and $\chi-2e=(1-e)^2$ show that every state with $|\delta p|>\delta e$ obeys $\delta e<(1+e)^2|\delta a|$, that is, $a\delta e<(1+e)^2|\Delta a|$.
Orientation-independent separation is therefore restricted to formations of the same order as the semimajor-axis difference.

Larger formations are instead centered by matching semilatus recta, $\delta p=0$, which removes the term $a\delta p/(kq)$ from $\rho_B$.
For $e>0$, Eq.~\eqref{eq:semilatus-offset} shows that centering requires
\begin{equation}
\delta p = 0 \quad \Leftrightarrow \quad
    E_\parallel=\frac{\eta^4}{2e}\delta a.
\label{eq:centering-condition}
\end{equation}
Near the circular limit, the $1/e$ scaling makes it preferable to null a residual $\delta a$ rather than to apply this condition.

Figure~\ref{fig:frame-drift-centered} illustrates centering for the example of Sec.~\ref{sec:vnb-geometry}.
With the same \SI{100}{\metre} offset as Fig.~\ref{fig:frame-drift-vnb}, but with $\Delta e$ chosen so that $\delta p=0$, which requires $aE_\parallel=\SI{88}{\metre}$, the phase/energy center lies on the $V$ axis while the deputy drifts along $V$.
The remaining binormal motion is the harmonic $\delta\boldsymbol e$ term, amplified by $1/(kq)$ near apoapsis.

\subsubsection{Phase-Robust Separation Bound}

Peters and Noomen bounded the closest approach of a safe-orbit family conservatively by its inner boundary at perigee, where both scale factors take their largest value, and evaluated that boundary for parallel relative eccentricity and inclination vectors~\cite{peters2021cotangential}.
The following result gives the corresponding bound in closed form for arbitrary vector orientation and adds the margin consumed by a semimajor-axis residual.
\begin{proposition}[Phase-robust BN-plane separation]
\label{prop:bn-bound}
Consider Keplerian relative motion about a chief with $0\le e<1$ and $\delta p=0$, and let $\widetilde{\delta\boldsymbol e}=\delta\boldsymbol e/(1+e)$, with magnitude $\widetilde{\delta e}$.
Then, for every initial longitude $\delta\lambda_0$ and all $t\ge t_0$, the separation satisfies
\begin{equation}
\|\boldsymbol\rho(t)\|\ge d_{\mathrm{BN}}\ge
d_{\mathrm{BN}}^{\mathrm{lb}}:=
\frac{\sqrt2\,a}{1+e}\,
\frac{|\widetilde{\delta\boldsymbol e}^{\top}\delta\boldsymbol i|}
{\sqrt{\widetilde{\delta e}^2+\delta i^2
+\|\widetilde{\delta\boldsymbol e}+\delta\boldsymbol i\|\,
\|\widetilde{\delta\boldsymbol e}-\delta\boldsymbol i\|}}.
\label{eq:bn-bound}
\end{equation}
The second inequality becomes an equality at $e=0$, where Eq.~\eqref{eq:bn-bound} is the near-circular minimum distance in the RN plane~\cite{damico2006ei}.
For given $\delta e$ and $\delta i$, $d_{\mathrm{BN}}^{\mathrm{lb}}$ is largest when $\delta\boldsymbol e\parallel\delta\boldsymbol i$, for which the closest approach of the perigee inner boundary of Ref.~\cite{peters2021cotangential} is
\begin{equation}
d_{\mathrm{BN}}^{\mathrm{lb}}
=a\min\left\{\frac{\delta e}{(1+e)^2},\,\frac{\delta i}{1+e}\right\}.
\label{eq:bn-bound-parallel}
\end{equation}
If a semimajor-axis residual $\Delta a_{\mathrm{err}}$ is subsequently added, then $\rho_B$ shifts by $\eta^2\Delta a_{\mathrm{err}}/q$, and for every $\delta\lambda_0$ and all $t\ge t_0$ the separation satisfies
\begin{equation}
d_{\mathrm{BN}}\ge d_{\mathrm{BN}}^{\mathrm{lb}}-(1+e)|\Delta a_{\mathrm{err}}|.
\label{eq:bn-physical-da-residual-bound}
\end{equation}
\end{proposition}
\begin{proof}
Under the Keplerian dynamics, only $\delta\lambda$ varies, which does not enter $\rho_B$ or $\rho_N$ in Eq.~\eqref{eq:position}.
The BN curve is therefore a fixed $2\pi$-periodic function of $\theta$, so $\|\boldsymbol\rho(t)\|\ge\sqrt{\rho_B^2+\rho_N^2}\ge d_{\mathrm{BN}}$ for every $\delta\lambda_0$ and all $t$.
With $\delta p=0$, Eq.~\eqref{eq:position}, $k,q\le1+e$, and a double-angle expansion give
\begin{equation*}
\begin{aligned}
\frac{\rho_B^2+\rho_N^2}{a^2}
&=\frac{\delta e^2\cos^2(\theta-\phi_e)}{k^2q^2}+\frac{\delta i^2\sin^2(\theta-\phi_i)}{k^2} \\
& \ge\frac{\widetilde{\delta e}^2\cos^2(\theta-\phi_e)+\delta i^2\sin^2(\theta-\phi_i)}{(1+e)^2}\\
&=\frac{\widetilde{\delta e}^2+\delta i^2+\widetilde{\delta e}^2\cos2(\theta-\phi_e)-\delta i^2\cos2(\theta-\phi_i)}{2(1+e)^2}\\
&\ge\frac{\widetilde{\delta e}^2+\delta i^2-\|\widetilde{\delta\boldsymbol e}+\delta\boldsymbol i\|\,\|\widetilde{\delta\boldsymbol e}-\delta\boldsymbol i\|}{2(1+e)^2}
=\left(\frac{d_{\mathrm{BN}}^{\mathrm{lb}}}{a}\right)^2.
\end{aligned}
\end{equation*}
The two cosines share the frequency $2\theta$ and differ in phase by $2(\phi_i-\phi_e)$, so they combine into a single sinusoid of amplitude $R$ with
$R^2=\widetilde{\delta e}^4+\delta i^4-2\widetilde{\delta e}^2\delta i^2\cos2(\phi_i-\phi_e)=(\widetilde{\delta e}^2+\delta i^2)^2-4(\widetilde{\delta\boldsymbol e}^{\top}\delta\boldsymbol i)^2=\|\widetilde{\delta\boldsymbol e}+\delta\boldsymbol i\|^2\|\widetilde{\delta\boldsymbol e}-\delta\boldsymbol i\|^2$.
The second inequality replaces this sinusoid by its minimum $-R$.
The first inequality is an equality when $e=0$, which gives $d_{\mathrm{BN}}=d_{\mathrm{BN}}^{\mathrm{lb}}$.
The same amplitude shows that the bound increases with $|\widetilde{\delta\boldsymbol e}^{\top}\delta\boldsymbol i|$ at fixed magnitudes and reduces to Eq.~\eqref{eq:bn-bound-parallel} for parallel vectors.
For a residual in $a$ alone, Eqs.~\eqref{eq:roe-definition}, \eqref{eq:classical-definitions}, and~\eqref{eq:semilatus-offset} give $\Delta(\delta\boldsymbol e)=-\eta^2(\Delta a_{\mathrm{err}}/a)\boldsymbol e$ and $\Delta(\delta p)=\eta^2\Delta a_{\mathrm{err}}/a$, so Eq.~\eqref{eq:position}, $k=1+e_x\cos\theta+e_y\sin\theta$, and $q\ge1-e$ give
\begin{equation*}
|\Delta\rho_B|
=\frac{a\,|\Delta(\delta p)-\Delta(\delta e_x)\cos\theta-\Delta(\delta e_y)\sin\theta|}{kq}
=\frac{\eta^2|\Delta a_{\mathrm{err}}|}{q}
\le(1+e)|\Delta a_{\mathrm{err}}|.
\end{equation*}
The normal component $\rho_N$ is unchanged.
The reverse triangle inequality then gives Eq.~\eqref{eq:bn-physical-da-residual-bound}.
\end{proof}
The bound replaces $k$ and $q$ by their periapsis value $1+e$ and is therefore conservative away from periapsis.
For parallel vectors, the BN distance equals $a\delta e/(kq)$ at the zeros of $\rho_N$, $\theta\in\{\phi_e,\phi_e+\pi\}$, and $a\delta i/k$ at the zeros of $\rho_B$, $\theta=\phi_e\pm\pi/2$.
One member of each antipodal pair has $\cos\nu\ge0$, and hence $k\ge1$ and $q\ge\sqrt\chi$, which gives
\begin{equation}
d_{\mathrm{BN}}\le a\min\left\{\frac{\delta e}{\sqrt\chi},\,\delta i\right\},
\qquad
1\le\frac{d_{\mathrm{BN}}}{d_{\mathrm{BN}}^{\mathrm{lb}}}\le\frac{(1+e)^2}{\sqrt\chi}.
\label{eq:bn-sharpness}
\end{equation}
The factor $(1+e)^2/\sqrt\chi$ increases monotonically from 1 at $e=0$ toward $2\sqrt2$ as $e\to1$, and it is approached by $\delta a=0$ designs with $\delta i\gg\delta e$.
For $\delta a=0$ and $e>0$, Eq.~\eqref{eq:centering-condition} gives $E_\parallel=0$, so $\phi_e=\omega\pm\pi/2$ and the zeros of $\rho_B$ fall at periapsis and apoapsis.
Equation~\eqref{eq:bn-bound-parallel} is then exact whenever $\delta e\ge(1+e)\,\delta i$, as in the formations of Sec.~\ref{sec:passive-safety-case-study}.

The residual estimate of Proposition~\ref{prop:bn-bound} extends to any perturbation that acts on $\rho_B$ alone.
An unmodeled semilatus-rectum error $\delta p_{\mathrm{err}}$, with $\delta\boldsymbol e$ and $\delta\boldsymbol i$ held fixed, shifts $\rho_B$ by $a\delta p_{\mathrm{err}}/(kq)$, and $kq\ge(1-e)^2$ gives
\begin{equation}
d_{\mathrm{BN}}\ge d_{\mathrm{BN}}^{\mathrm{lb}}-\frac{a|\delta p_{\mathrm{err}}|}{(1-e)^2}.
\label{eq:bn-residual-bound}
\end{equation}
A physical residual produces $\delta p_{\mathrm{err}}=\eta^2\Delta a_{\mathrm{err}}/a$, for which Eq.~\eqref{eq:bn-residual-bound} would give the weaker margin $(1+e)|\Delta a_{\mathrm{err}}|/(1-e)$.
Equation~\eqref{eq:bn-physical-da-residual-bound} is sharper because the accompanying change in $\delta\boldsymbol e$ cancels the $1/k$ dependence of the shift.
In either case, the clearance allowance for an unknown energy residual is independent of the accumulated phase.

For formations whose scale exceeds the expected energy error, set $\delta p=0$ to center the nominal BN geometry.
Choose $\delta\boldsymbol e\parallel\delta\boldsymbol i$ and size the vectors through Eq.~\eqref{eq:bn-bound-parallel}.
Reserve the margin $(1+e)|\Delta a_{\mathrm{err}}|$ of Eq.~\eqref{eq:bn-physical-da-residual-bound} for semimajor-axis uncertainty.
This mirrors near-circular E/I-vector separation, where the nominal clearance above the keep-out radius is sized to accommodate semimajor-axis uncertainty.
As $e\to0$, the BN plane becomes the RN plane and the reserved margin reduces to $|\Delta a_{\mathrm{err}}|$.
Section~\ref{sec:roe-stm} shows why the relative-vector orientation must also be monitored over longer, perturbed horizons.

\section{State Transition Matrices and Control Input Matrix}
\label{sec:stm-control}

This section derives the Keplerian and secular-$J_2$ STMs and the impulsive control input matrix of the proposed EROE.

\subsection{State Transition Matrices}
\label{sec:roe-stm}

Under Keplerian motion, the EROE rates are
\begin{equation}
\dot{\delta\lambda}=-\frac32n\delta a,\qquad
\dot{\delta a}=\dot{\delta e_x}=\dot{\delta e_y}
=\dot{\delta i_x}=\dot{\delta i_y}=0.
\end{equation}
Integration gives the STM
\begin{equation}
\delta\boldsymbol\alpha(t)=\boldsymbol\Phi_K(\tau)\delta\boldsymbol\alpha(t_0),
\qquad
\boldsymbol\Phi_K(\tau)=
\begin{bmatrix}
1&0&0&0&0&0\\
-\tfrac32n\tau&1&0&0&0&0\\
0&0&1&0&0&0\\
0&0&0&1&0&0\\
0&0&0&0&1&0\\
0&0&0&0&0&1
\end{bmatrix}.
\end{equation}
This STM is identical to the qnsROE form.
The C-set also has a single drifting element, but its rate depends on both $C_1$ and $C_2$~\cite{peters2021cotangential}.
The Keplerian STM of the Delurgio--D'Amico EROE also contains secular terms in $\delta\boldsymbol e^{(D)}$~\cite{delurgio2024}.

\subsubsection{$J_2$-perturbed STM and $J_2$-invariant formation design}

Ref.~\cite{koenig2017stm} provides a convenient way to derive the STM from the the variational equations of the classical orbital elements and the linear mapping to the EROE.
In this subsection, the secular $J_2$ effects are introduced as an example to demonstrate the derivation of the STM in the proposed EROE space.
The first-order secular-$J_2$ mean-element rates are~\cite{schaub2001j2}
\begin{subequations}
\begin{equation}
\dot a=\dot e=\dot i=0, \quad
\dot\Omega =-2\kappa c_i, \quad
\dot\omega =\kappa(5c_i^2-1), \quad
\dot M =n+\kappa\eta (3c_i^2-1).
\label{eq:j2-chief-rates}
\end{equation}
The chief inclination factors and the secular coefficient are defined by
\begin{equation}
c_i=\cos i,\qquad s_i=\sin i,\qquad
\kappa=\frac34J_2n\left(\frac{R_b}{a}\right)^2\eta^{-4},
\end{equation}
\end{subequations}
where $J_2$ is the second zonal harmonic and $R_b$ is the radius of the central body.
At $e=0$, $\dot\omega$ denotes the eccentricity-vector rotation rate.
From Eqs.~\eqref{eq:roe-definition} and\eqref{eq:j2-chief-rates}, the secular-$J_2$ model keeps $\Delta a$, $\Delta e$, and $\Delta i$ constant.
With $\Delta\psi:=\Delta\omega+c_i\Delta\Omega$, linearizing Eq.~\eqref{eq:j2-chief-rates} and differentiating Eq.~\eqref{eq:roe-definition} gives
\begin{equation}
\begin{aligned}
\ell&:=\frac{\Delta\kappa}{\kappa}=-\frac72\frac{\Delta a}{a}+\frac{4e\Delta e}{\eta^2},\\
\sigma_\Omega&:=\Delta\dot\Omega=\kappa(-2c_i\ell+2s_i\Delta i),\\
\sigma_\psi&:=\Delta\dot\psi=\Delta\dot\omega+c_i\Delta\dot\Omega=\kappa[(3c_i^2-1)\ell-8c_is_i\Delta i],\\
\sigma_M&:=\Delta\dot M=-\frac32n\frac{\Delta a}{a}
+\kappa\eta\left[(3c_i^2-1)\left(\ell-\frac{e\Delta e}{\eta^2}\right)-6c_is_i\Delta i\right],\\
\sigma_\lambda&:=\dot{\delta\lambda}=\sigma_M+\frac{\eta^3}{\chi}\sigma_\psi,\qquad
\dot E_\perp=\eta^2e\,\sigma_\psi,\qquad
\dot{\delta i_y}=\eta^2s_i\sigma_\Omega.
\end{aligned}
\label{eq:j2-rates}
\end{equation}
To express these rates directly in the proposed EROE, let $\delta\boldsymbol\alpha_0:=\delta\boldsymbol\alpha(t_0)$, $D=3c_i^2-1$, $e_{x0}=e\cos\omega_0$, and $e_{y0}=e\sin\omega_0$.
With all chief quantities held fixed at $t_0$, the longitude-, precession-, and nodal-rate Jacobians with respect to the initial EROE are
\begin{equation}
\begin{aligned}
\boldsymbol L&:=\frac{\partial\sigma_\lambda}{\partial(\delta\boldsymbol\alpha_0)}
=\begin{bmatrix}
-\dfrac32n-\dfrac{\kappa\eta D(7+e^4)}{\chi^2}&0&
\dfrac{\kappa D(7-e^2)e_{x0}}{\eta\chi^2}&
\dfrac{\kappa D(7-e^2)e_{y0}}{\eta\chi^2}&
-\dfrac{2\kappa c_is_i(7-e^2)}{\eta\chi}&0
\end{bmatrix},\\[4pt]
\boldsymbol P&:=\frac{\partial\sigma_\psi}{\partial(\delta\boldsymbol\alpha_0)}
=\begin{bmatrix}
-\dfrac{\kappa D(7-e^2)}{2\chi}&0&
\dfrac{4\kappa D e_{x0}}{\eta^2\chi}&
\dfrac{4\kappa D e_{y0}}{\eta^2\chi}&
-\dfrac{8\kappa c_is_i}{\eta^2}&0
\end{bmatrix},\\[4pt]
\boldsymbol O&:=\frac{\partial\sigma_\Omega}{\partial(\delta\boldsymbol\alpha_0)}
=\begin{bmatrix}
\dfrac{\kappa c_i(7-e^2)}{\chi}&0&
-\dfrac{8\kappa c_i e_{x0}}{\eta^2\chi}&
-\dfrac{8\kappa c_i e_{y0}}{\eta^2\chi}&
\dfrac{2\kappa s_i}{\eta^2}&0
\end{bmatrix}.
\end{aligned}
\label{eq:j2-coefficient-rows}
\end{equation}
Let $L_j$, $P_j$, and $O_j$ denote the $j$th entries of the coefficient rows in Eq.~\eqref{eq:j2-coefficient-rows}.
In the basis whose eccentricity pair co-rotates with the chief periapsis, these coefficients define the constant shear matrix
\begin{equation}
\boldsymbol B=
\begin{bmatrix}
0&0&0&0&0&0\\
L_1&0&L_3&L_4&L_5&0\\
-\eta^2e_{y0}P_1&0&-\eta^2e_{y0}P_3&-\eta^2e_{y0}P_4&-\eta^2e_{y0}P_5&0\\
\eta^2e_{x0}P_1&0&\eta^2e_{x0}P_3&\eta^2e_{x0}P_4&\eta^2e_{x0}P_5&0\\
0&0&0&0&0&0\\
\eta^2s_iO_1&0&\eta^2s_iO_3&\eta^2s_iO_4&\eta^2s_iO_5&0
\end{bmatrix}.
\label{eq:j2-shear}
\end{equation}
Rows 1 and 5 and columns 2 and 6 are identically zero.
Since $\boldsymbol B^2=\boldsymbol 0$, its exponential is $\exp(\tau\boldsymbol B)=\boldsymbol I_6+\tau\boldsymbol B$.
Considering the rotation of the eccentricity vector due to the $J_2$ perturbation \cite{koenig2017stm}, the closed-form STM is given as
\begin{equation}
\boldsymbol\Phi_{J_2}(\tau)=
\operatorname{diag}\!\left(1,1,\boldsymbol R(\dot\omega\tau),1,1\right)
\left(\boldsymbol I_6+\tau\boldsymbol B\right),\qquad \tau=t-t_0.
\label{eq:j2-stm}
\end{equation}
Equation~\eqref{eq:j2-stm} is exact within the first-order secular mean-element model and reduces to $\boldsymbol\Phi_K(\tau)$ when $J_2=0$.

The derived closed-form STM provides an insight into the $J_2$-invariant formation designs.
Schaub and Alfriend's nodal and mean-argument-of-latitude rate-matching conditions~\cite{schaub2001j2}, with the $J_2\delta a_0$ terms of Eq.~\eqref{eq:j2-rates} retained, are
\begin{equation}
\begin{aligned}
\sigma_\Omega&=0,\\
\sigma_u:=\Delta\dot u
&=\sigma_M+\sigma_\psi-c_i\sigma_\Omega=0.
\end{aligned}
\label{eq:j2-schaub-alfriend}
\end{equation}
The term $-c_i\sigma_\Omega$ removes the nodal contribution, and $\sigma_u=0$ can accommodate a small nonzero $\delta a_0$.

For $\delta a_0=0$ and a non-equatorial chief, these conditions reduce to $e\Delta e=\Delta i=0$, which is equivalent to
\begin{equation}
\delta a_0=\delta p_0=\delta i_{x0}=0.
\label{eq:j2-stable-design}
\end{equation}
For $e>0$, this requires equal mean eccentricities and inclinations.
At $e=0$, the first-order eccentricity constraint vanishes, leaving the qnsROE condition $\delta i_{x0}=0$ \cite{damico2010thesis}.
Nonetheless, Eq.~\eqref{eq:j2-stable-design} still leaves $\delta\boldsymbol e$ rotating at $\dot\omega$.
As in near-circular formations, maintaining the E/I-vector separation margin of Eq.~\eqref{eq:bn-bound} over long horizons generally requires station keeping of $\delta\boldsymbol e$~\cite{damico2010thesis}.

\subsection{Control Input Matrix}
\label{sec:control}

The input matrix can be obtained from the variational equations~\cite{brouwer1959, vallado2001fundamentals}, or equivalently by inverting Eqs.~\eqref{eq:position} and~\eqref{eq:velocity} for a velocity impulse at fixed position.
An impulsive deputy velocity increment resolved in the chief VNB frame changes the elements by
\begin{equation}
\Delta(\delta\boldsymbol\alpha)=\boldsymbol\Gamma(\theta)
\begin{bmatrix}\Delta v_V \\ \Delta v_N \\ \Delta v_B\end{bmatrix},
\quad
\boldsymbol\Gamma(\theta)=\frac1{na}
\begin{bmatrix}
2q/\eta&0&0\\
-2\mathcal C
&0&-\dfrac{2\eta^4}{\chi kq}\\
\dfrac{2\eta^3\cos\theta}{q}&0
&\dfrac{\eta^3[\chi\sin\theta+2e_y]}{kq}\\
\dfrac{2\eta^3\sin\theta}{q}&0
&-\dfrac{\eta^3[\chi\cos\theta+2e_x]}{kq}\\
0&\dfrac{\eta^3\cos\theta}{k}&0\\
0&\dfrac{\eta^3\sin\theta}{k}&0
\end{bmatrix}.
\label{eq:control}
\end{equation}
Evidently, the map is decoupled between the in-plane and out-of-plane dynamics for all $0\le e<1$.
This structure enables straightforward application of the closed-form impulsive-control algorithm developed in Refs.~\cite{chernick2018,chernick2021}, which has previously been applied to the Delurgio--D'Amico EROE~\cite{delurgio2024}.

Through Eq.~\eqref{eq:map-cset}, the V and B columns correspond to the TAN-frame input matrix of the C-set~\cite{peters2021cotangential}.
Retaining $\delta a$ as the first EROE isolates the energy mode, which can be corrected by $\Delta v_V$.
For $e>0$, a subsequent $\Delta v_B$ away from the apsides can adjust $\delta p$ without changing $\delta a$.
Although $\delta p\to\delta a$ as $e\to0$, using it as the first EROE would preserve in-plane/out-of-plane decoupling, as in the C-set, but make that element respond to both impulses.
At $e=0$, where $\theta=u$, $\hat V=\hat T$, and $\hat B=\hat R$, ordering the columns as $[T,N,R]$ reduces Eq.~\eqref{eq:control} to the qnsROE input matrix~\cite{d2005relative}
\begin{equation}
\boldsymbol\Gamma_0(u)=\frac1{na}
\begin{bmatrix}
2&0&0\\0&0&-2\\
2\cos u&0&\sin u\\
2\sin u&0&-\cos u\\
0&\cos u&0\\0&\sin u&0
\end{bmatrix}.
\end{equation}

\section{Direct Cartesian Derivation in the VNB Frame}
\label{sec:cartesian-derivation}

The elements of Sec.~\ref{sec:new-eroes} were constructed from orbital-element differences.
This section provides an independent derivation by solving the Cartesian VNB variational equations directly, without assuming an element-to-Cartesian map.
The derivation yields the VNB counterpart of the YA Cartesian STM and tests whether a physically normalized modal solution selects the proposed EROE, in the same way that the Delurgio--D'Amico EROE coincide with the YA integration constants in the RTN frame~\cite{delurgio2024}.
Closed-form maps between the C-set and the TAN-frame Cartesian state are also available~\cite{peters2021cotangential} and, by Eq.~\eqref{eq:map-cset}, are equivalent to the maps obtained here.
Following the structure of the YA derivation~\cite{yamanaka2002}, the Cartesian variables are scaled, the six modal responses are assembled into a fundamental matrix $\bar{\boldsymbol F}_{\mathrm{VNB}}$, and its inverse action is obtained in closed form.
Section~\ref{sec:cartesian-stm} assembles the physical Cartesian STM, and Sec.~\ref{sec:cartesian-identification} then identifies the modal constants with the EROE of Sec.~\ref{sec:new-eroes} and establishes the equivalence of the two formulations.

\subsection{Relative Equations of Motion in the VNB Frame}
\label{sec:cartesian-eom}

For the VNB state $\boldsymbol x_{\mathrm{VNB}}$ of Eq.~\eqref{eq:relative-states}, the chief radial direction and the frame angular velocity resolved in VNB are $\hat{\boldsymbol R}^{V}=[S/q,0,k/q]^{\top}$, the first column of $\boldsymbol R_{\mathrm{VR}}$ in Eq.~\eqref{eq:rtn-to-vnb}, and $\boldsymbol\omega_V=[0,\omega_V,0]^{\top}$.
Linearizing the two-body relative equation of motion in this rotating frame gives
\begin{equation}
\ddot{\boldsymbol\rho}^{V}
=\left[
\frac{\mu}{r_c^3}\left(3\hat{\boldsymbol R}^{V}
(\hat{\boldsymbol R}^{V})^{\top}-\boldsymbol I_3\right)
-[\dot{\boldsymbol\omega}_V]_\times
-[\boldsymbol\omega_V]_\times^2
\right]\boldsymbol\rho^{V}
-2[\boldsymbol\omega_V]_\times\dot{\boldsymbol\rho}^{V}.
\label{eq:vnb-cartesian-variational}
\end{equation}

\subsection{Scaling of the Variables}
\label{sec:cartesian-scaling}

Let $(\cdot)'$ denote differentiation with respect to $\theta$.
As in the YA derivation~\cite{yamanaka2002}, scaling the coordinates by the length $\bar r:=a/k=r_c/\eta^2$ gives
\begin{equation}
\bar X=\frac{\rho_V}{\bar r},\qquad
\bar Y=\frac{\rho_N}{\bar r},\qquad
\bar Z=\frac{\rho_B}{\bar r}.
\label{eq:vnb-radius-scaled-state}
\end{equation}
This normalizes the coordinates by the chief radius up to the constant factor $\eta^2$.
Referring the scaling to $\bar r$ rather than to $r_c$ keeps that factor out of the de-scaling matrices below.
With $\dot\theta$ and $\omega_V$ of Eq.~\eqref{eq:chief-notation}, the change of independent variable uses the identities
\begin{equation}
\begin{aligned}
\frac{d}{dt}&=\dot\theta\frac{d}{d\theta},
&\frac{\ddot\theta}{\dot\theta^2}&=-\frac{2S}{k},\\
\frac{\mu/r_c^3}{\dot\theta^2}&=\frac1k,
&\frac{\dot\omega_V}{\dot\theta^2}&=\frac{S(2k-3q^2)}{q^4}.
\end{aligned}
\end{equation}
Because $k'=-S$, the scaling removes the first-derivative term $(\ddot\theta/\dot\theta^2)\boldsymbol\rho'$ introduced by the change of independent variable, and substitution of Eq.~\eqref{eq:vnb-radius-scaled-state} into Eq.~\eqref{eq:vnb-cartesian-variational} gives
\begin{equation}
\begin{aligned}
\bar X''={}&
\left(\frac{3S^2}{kq^2}-1+\frac{k^2}{q^4}\right)\bar X
+\frac{2S(2q^2-k)}{q^4}\bar Z
-\frac{2k}{q^2}\bar Z',\\
\bar Y''+\bar Y={}&0,\\
\bar Z''={}&
\frac{2S(q^2+k)}{q^4}\bar X
+\left(\frac{3k}{q^2}-1+\frac{k^2}{q^4}\right)\bar Z
+\frac{2k}{q^2}\bar X'.
\end{aligned}
\label{eq:vnb-radius-scaled-dynamics}
\end{equation}
The scaled state is
\begin{equation}
\bar{\boldsymbol x}_{\mathrm{VNB}}=
\begin{bmatrix}
\bar X&\bar Y&\bar Z&\bar X'&\bar Y'&\bar Z'
\end{bmatrix}^{\top}.
\label{eq:radius-scaled-state-vector}
\end{equation}
The out-of-plane equation has reduced to the harmonic oscillator $\bar Y''+\bar Y=0$, identical to the derivation of the radius-normalized YA STM.

The physical state is recovered by differentiating $[\rho_V,\rho_N,\rho_B]^{\top}=\bar r\,[\bar X,\bar Y,\bar Z]^{\top}$, which gives
\begin{equation}
\boldsymbol x_{\mathrm{VNB}}(t)=\boldsymbol Q(t)\bar{\boldsymbol x}_{\mathrm{VNB}}(t),
\qquad
\boldsymbol Q(t)=
\begin{bmatrix}
\bar r\boldsymbol I_3&\boldsymbol 0\\
\dot{\bar r}\boldsymbol I_3&\bar r\dot\theta \boldsymbol I_3
\end{bmatrix},
\qquad
\dot{\bar r}=\frac{aS\dot\theta}{k^2}=\frac{anS}{\eta^3}.
\label{eq:descaling}
\end{equation}
Using $\dot{\bar r}/(\bar r^2\dot\theta)=S/a$, the block-triangular inverse is
\begin{equation}
\boldsymbol Q^{-1}(t)=
\frac1a
\begin{bmatrix}
k\boldsymbol I_3&\boldsymbol 0\\
-S\boldsymbol I_3&
\dfrac{k}{\dot\theta}\boldsymbol I_3
\end{bmatrix},
\qquad
\frac{k}{\dot\theta}=\frac{\eta^3}{nk}.
\label{eq:descaling-inverse}
\end{equation}

\subsection{Integration Constants and Their Responses}
\label{sec:cartesian-fsm}

The solution separates into two normal harmonic modes, two coupled in-plane harmonic modes, one phase mode, and one energy mode.
Solving the decoupled oscillator in Eq.~\eqref{eq:vnb-radius-scaled-dynamics} gives the normal pair
\begin{equation}
\bar Y=K_{Nx}\sin\theta-K_{Ny}\cos\theta,
\qquad
\bar Y'=K_{Nx}\cos\theta+K_{Ny}\sin\theta.
\label{eq:vnb-normal-response}
\end{equation}
The coupled $(\bar X,\bar Z)$ subsystem has the conserved amplitude
\begin{equation}
A=\frac{2}{\eta^4q}
\left[kq^2\bar X'+S(q^2+k)\bar X+2k^2\bar Z\right].
\label{eq:vnb-energy-integral}
\end{equation}
The equivalent first integral is
\begin{equation}
\bar X'=w_X\bar X+w_Z\bar Z+w_AA,
\quad
w_X=-\frac{S(q^2+k)}{kq^2},
\quad
w_Z=-\frac{2k}{q^2},
\quad
w_A=\frac{\eta^4}{2kq}.
\label{eq:vnb-radius-scaled-first-integral}
\end{equation}
Defining the binormal variable $b=q\bar Z$ and substituting Eq.~\eqref{eq:vnb-radius-scaled-first-integral} into the last equation of Eq.~\eqref{eq:vnb-radius-scaled-dynamics} yields
\begin{equation}
\begin{aligned}
b''+b&=P,
&b&=P-K_x\cos\theta-K_y\sin\theta,\\
\chi P&=\eta^4A-2(e_xK_x+e_yK_y),
&P'&=0.
\end{aligned}
\label{eq:vnb-reduced-cartesian-modes}
\end{equation}
The in-plane harmonic pair is defined by
\begin{equation}
g=K_x\cos\theta+K_y\sin\theta,
\qquad
h=K_x\sin\theta-K_y\cos\theta.
\label{eq:vnb-harmonic-pair}
\end{equation}
The binormal response is
\begin{equation}
\bar Z=\frac{P-g}{q},
\qquad
\bar Z'=\frac{h}{q}+\frac{S}{q^2}\bar Z.
\label{eq:vnb-binormal-response}
\end{equation}
Integrating the first integral of Eq.~\eqref{eq:vnb-radius-scaled-first-integral}, whose homogeneous solution is $kq$, gives the velocity-direction response
\begin{equation}
\bar X=k\left[
\mathcal C A+\frac q\eta\Lambda+\frac{2h}{\chi q}
\right],
\qquad
\Lambda'=-\frac{3\eta^3}{2k^2}A,
\qquad
\Lambda(t)=\Lambda_0-\tfrac32n\tau A,
\label{eq:vnb-cartesian-responses}
\end{equation}
where $\mathcal C$ is the velocity-direction coefficient of Eq.~\eqref{eq:position}.
The last equality follows from $dM/d\theta=\eta^3/k^2$ and $M-M_0=n\tau$.
The secular shear is thus obtained from the Cartesian equations rather than assumed from an ROE propagation law.

The constants are normalized using only Cartesian modal properties.
The conventions of Eqs.~\eqref{eq:vnb-normal-response} and~\eqref{eq:vnb-reduced-cartesian-modes} make $K_x$ and $K_y$ the unit coefficients of $-\cos\theta$ and $-\sin\theta$ in the homogeneous part of $b=q\bar Z$, and $K_{Nx}$ and $K_{Ny}$ the unit coefficients of $\sin\theta$ and $-\cos\theta$ in $\bar Y$.
The phase variable is normalized so that $\Lambda/n$ is the first-order time offset.
Its physical displacement is then $v_c\Lambda/n$ along $\hat V$, which fixes the factor $q/\eta$ in Eq.~\eqref{eq:vnb-cartesian-responses}.
The conserved amplitude $A$ is normalized by the secular shear rate $\dot\Lambda=-\tfrac32nA$, which fixes the prefactor in Eq.~\eqref{eq:vnb-energy-integral}.
Finally, the bounded in-plane modes are assigned no free time-offset constant, and their unit amplitudes and the Cartesian equations then fix the term $2h/(\chi q)$ in Eq.~\eqref{eq:vnb-cartesian-responses}.
Together, Eqs.~\eqref{eq:vnb-normal-response}--\eqref{eq:vnb-cartesian-responses} provide the six independent integration constants
\begin{equation}
\boldsymbol c_{\mathrm{VNB}}=
\begin{bmatrix}
A&\Lambda_0&K_x&K_y&K_{Nx}&K_{Ny}
\end{bmatrix}^{\top}.
\label{eq:vnb-integration-constants}
\end{equation}

\subsection{Closed-Form Cartesian VNB State Transition Matrix}
\label{sec:cartesian-stm}

\subsubsection{Explicit Fundamental Matrix}

The fundamental matrix is written at the running time $t$, whose true argument of latitude $\theta$ follows from Kepler's equation, with elapsed time $\tau=t-t_0$ as in Eq.~\eqref{eq:chief-notation}.
Every anomaly-dependent scalar it contains has already been introduced: $k$, $S$, and $q$ of Eq.~\eqref{eq:chief-notation}, the coefficient $\mathcal C$ of Eq.~\eqref{eq:position}, and the first-integral weights $w_X$, $w_Z$, and $w_A$ of Eq.~\eqref{eq:vnb-radius-scaled-first-integral}.
A subscript $0$ on any of these scalars denotes evaluation at the reference epoch $t_0$, where $\tau=0$.
Collecting the response of each constant in Eq.~\eqref{eq:vnb-integration-constants} as a column, in the order $[A,\Lambda_0,K_x,K_y,K_{Nx},K_{Ny}]$, gives the radius-scaled fundamental matrix
\begin{equation}
\bar{\boldsymbol x}_{\mathrm{VNB}}(t)=\bar{\boldsymbol F}_{\mathrm{VNB}}(t,t_0)\boldsymbol c_{\mathrm{VNB}},
\qquad
\bar{\boldsymbol F}_{\mathrm{VNB}}(t,t_0)=
\begin{bmatrix}
\bar F_{11}&\bar F_{12}&\bar F_{13}&\bar F_{14}&0&0\\
0&0&0&0&\bar F_{25}&\bar F_{26}\\
\bar F_{31}&0&\bar F_{33}&\bar F_{34}&0&0\\
\bar F_{41}&\bar F_{42}&\bar F_{43}&\bar F_{44}&0&0\\
0&0&0&0&\bar F_{55}&\bar F_{56}\\
\bar F_{61}&0&\bar F_{63}&\bar F_{64}&0&0
\end{bmatrix}.
\label{eq:response-matrix-explicit}
\end{equation}
Its nonzero entries, all evaluated at $\theta$, are
\begin{equation}
\begin{aligned}
\bar F_{11}&=k\left(\mathcal C-\frac{3n\tau q}{2\eta}\right),
&\bar F_{12}&=\frac{kq}{\eta},\\
\bar F_{13}&=\frac{2ks_\theta}{\chi q},
&\bar F_{14}&=-\frac{2kc_\theta}{\chi q},\\
\bar F_{25}&=s_\theta,
&\bar F_{26}&=-c_\theta,\\
\bar F_{31}&=\frac{\eta^4}{\chi q},
&\bar F_{33}&=-\frac{c_\theta+2e_x/\chi}{q},
&\bar F_{34}&=-\frac{s_\theta+2e_y/\chi}{q},\\
\bar F_{41}&=w_X\bar F_{11}+w_Z\bar F_{31}+w_A,
&\bar F_{42}&=w_X\bar F_{12},\\
\bar F_{43}&=w_X\bar F_{13}+w_Z\bar F_{33},
&\bar F_{44}&=w_X\bar F_{14}+w_Z\bar F_{34},\\
\bar F_{55}&=c_\theta,
&\bar F_{56}&=s_\theta,\\
\bar F_{61}&=\frac{S\eta^4}{\chi q^3},\\
\bar F_{63}&=\frac{s_\theta}{q}
-\frac{S}{q^3}\left(c_\theta+\frac{2e_x}{\chi}\right),
&\bar F_{64}&=-\frac{c_\theta}{q}
-\frac{S}{q^3}\left(s_\theta+\frac{2e_y}{\chi}\right).
\end{aligned}
\label{eq:response-scalar-entries}
\end{equation}

\subsubsection{Explicit Inverse of the Fundamental Matrix}
\label{sec:cartesian-inverse}

Direct evaluation of Eq.~\eqref{eq:response-matrix-explicit} gives
\begin{equation}
\det\bar{\boldsymbol F}_{\mathrm{VNB}}(t,t_0)
=\frac{\eta^3}{2\chi}>0,
\qquad 0\leq e<1,
\label{eq:radius-scaled-response-determinant}
\end{equation}
which guarantees unique recovery of the six modal constants at any reference anomaly.
The inverse rows follow the constant order $[A,\Lambda_0,K_x,K_y,K_{Nx},K_{Ny}]$, and its columns follow the scaled-state order $[\bar X_0,\bar Y_0,\bar Z_0,\bar X_0',\bar Y_0',\bar Z_0']$.
Solving the modal responses of Eqs.~\eqref{eq:vnb-normal-response}--\eqref{eq:vnb-cartesian-responses} at $t_0$, where $\tau=0$, gives the analytical inverse
\begin{equation}
\begin{gathered}
\boldsymbol c_{\mathrm{VNB}}=\bar{\boldsymbol G}_0\,\bar{\boldsymbol x}_{\mathrm{VNB}}(t_0),
\qquad
\bar{\boldsymbol G}_0:=\bar{\boldsymbol F}_{\mathrm{VNB}}^{-1}(t_0,t_0),\\[4pt]
\bar{\boldsymbol G}_0=
\begin{bmatrix}
\bar G_{11}&0&\bar G_{13}&\bar G_{14}&0&0\\
\bar G_{21}&0&\bar G_{23}&\bar G_{24}&0&\bar G_{26}\\
\bar G_{31}&0&\bar G_{33}&\bar G_{34}&0&\bar G_{36}\\
\bar G_{41}&0&\bar G_{43}&\bar G_{44}&0&\bar G_{46}\\
0&s_{\theta0}&0&0&c_{\theta0}&0\\
0&-c_{\theta0}&0&0&s_{\theta0}&0
\end{bmatrix}.
\end{gathered}
\label{eq:radius-scaled-fundamental-inverse}
\end{equation}
Here, $c_{\theta0}=\cos\theta_0$ and $s_{\theta0}=\sin\theta_0$.
The scalar entries at the reference epoch, evaluated with the first-row entries computed before those of the second row, are
\begin{equation}
\begin{aligned}
\bar G_{11}&=\frac{2S_0(q_0^2+k_0)}{\eta^4q_0},
&\bar G_{13}&=\frac{4k_0^2}{\eta^4q_0},\\
\bar G_{14}&=\frac{2k_0q_0}{\eta^4},\\[3pt]
\bar G_{21}&=\frac{\eta}{k_0q_0}-\frac{\eta\mathcal C_0}{q_0}\bar G_{11},
&\bar G_{23}&=\frac{2\eta S_0}{\chi q_0^3}-\frac{\eta\mathcal C_0}{q_0}\bar G_{13},\\
\bar G_{24}&=-\frac{\eta\mathcal C_0}{q_0}\bar G_{14},
&\bar G_{26}&=-\frac{2\eta}{\chi q_0},\\[3pt]
\bar G_{31}&=\frac{2S_0(q_0^2+k_0)}{q_0^3}c_{\theta0},
&\bar G_{41}&=\frac{2S_0(q_0^2+k_0)}{q_0^3}s_{\theta0},\\
\bar G_{33}&=\left(\frac{\eta^2}{q_0}+\frac{2k_0^2}{q_0^3}\right)c_{\theta0}
-\frac{S_0}{q_0}s_{\theta0},
&\bar G_{43}&=\left(\frac{\eta^2}{q_0}+\frac{2k_0^2}{q_0^3}\right)s_{\theta0}
+\frac{S_0}{q_0}c_{\theta0},\\
\bar G_{34}&=\frac{2k_0}{q_0}c_{\theta0},
&\bar G_{44}&=\frac{2k_0}{q_0}s_{\theta0},\\
\bar G_{36}&=q_0s_{\theta0}-\frac{2S_0}{q_0}c_{\theta0},
&\bar G_{46}&=-q_0c_{\theta0}-\frac{2S_0}{q_0}s_{\theta0}.
\end{aligned}
\label{eq:radius-scaled-inverse-entries}
\end{equation}
The identities $q_0^2=k_0^2+S_0^2$ and $2-\chi=\eta^2$ simplify the in-plane harmonic coefficients to the forms above.

\subsubsection{Physical State Transition Matrix}

De-scaling Eq.~\eqref{eq:response-matrix-explicit} with Eq.~\eqref{eq:descaling} gives the physical fundamental matrix $\boldsymbol F_{\mathrm{VNB}}(t,t_0):=\boldsymbol Q(t)\bar{\boldsymbol F}_{\mathrm{VNB}}(t,t_0)$.
Its initial-state inverse is $\boldsymbol F_{\mathrm{VNB}}^{-1}(t_0,t_0)=\bar{\boldsymbol G}_0\boldsymbol Q^{-1}(t_0)$, with $\boldsymbol Q^{-1}$ given by Eq.~\eqref{eq:descaling-inverse}.
Eliminating $\boldsymbol c_{\mathrm{VNB}}$ between the initial and final Cartesian states gives the analytical STM in the VNB frame as
\begin{equation}
\boldsymbol{\Phi}_{\mathrm{VNB}}(t,t_0)
=\boldsymbol Q(t)\bar{\boldsymbol F}_{\mathrm{VNB}}(t,t_0)\,
\bar{\boldsymbol G}_0\boldsymbol Q^{-1}(t_0).
\label{eq:vnb-cartesian-stm}
\end{equation}
Note that every factor is nonsingular for $0\leq e<1$ by Eqs.~\eqref{eq:descaling} and~\eqref{eq:radius-scaled-response-determinant}.

\subsection{Identification with the Eccentric Relative Orbital Elements}
\label{sec:cartesian-identification}

Both $\boldsymbol c_{\mathrm{VNB}}$ and $\delta\boldsymbol\alpha(t_0)$ are complete sets of constants for the same linear system, so they must be related by a constant nonsingular matrix.
The remaining question is the form of that matrix under the intrinsic Cartesian normalization specified above.
Comparing the radius-scaled position responses of Eqs.~\eqref{eq:vnb-normal-response}--\eqref{eq:vnb-cartesian-responses} with $k$ times Eq.~\eqref{eq:position} gives
\begin{equation}
\boldsymbol c_{\mathrm{VNB}}
=\begin{bmatrix}A&\Lambda_0&K_x&K_y&K_{Nx}&K_{Ny}\end{bmatrix}^{\top}
=\begin{bmatrix}\delta a&\delta\lambda_0&\delta e_x&\delta e_y&\delta i_x&\delta i_y\end{bmatrix}^{\top}
=\delta\boldsymbol\alpha(t_0).
\label{eq:cartesian-identification}
\end{equation}
Under the normalization of Sec.~\ref{sec:cartesian-fsm}, the constant transformation between the two sets is therefore the identity.
The C-set corresponds to a different normalization of the same modes, related to this one by Eq.~\eqref{eq:map-cset}.
The same comparison identifies the particular-solution constant $P$ of Eq.~\eqref{eq:vnb-reduced-cartesian-modes} with the normalized semilatus-rectum difference $\delta p$ of Eq.~\eqref{eq:semilatus-offset}, which is why $\delta p$, and not $\delta a$, centers the BN geometry.

Let $\bar{\boldsymbol\Psi}_{\mathrm{VNB}}(t):=\bar{\boldsymbol F}_{\mathrm{VNB}}(t,t)$ denote the scaled instantaneous map, obtained by setting $\tau=0$ in the evaluated entries of Eq.~\eqref{eq:response-scalar-entries}.
Its physical counterpart is $\boldsymbol\Psi_{\mathrm{VNB}}(t):=\boldsymbol Q(t)\bar{\boldsymbol\Psi}_{\mathrm{VNB}}(t)$, the matrix form of Eqs.~\eqref{eq:position} and~\eqref{eq:velocity}, so $\boldsymbol x_{\mathrm{VNB}}(t)=\boldsymbol\Psi_{\mathrm{VNB}}(t)\,\delta\boldsymbol\alpha(t)$.
At any epoch, the instantaneous EROE can therefore be recovered from the physical state by applying $\boldsymbol Q^{-1}$ and using Eqs.~\eqref{eq:radius-scaled-fundamental-inverse} and~\eqref{eq:radius-scaled-inverse-entries} with that epoch as the reference.
The radius normalization remains confined to the auxiliary Cartesian state.

Using the EROE STM of Sec.~\ref{sec:roe-stm}, the physical fundamental matrix satisfies $\boldsymbol F_{\mathrm{VNB}}(t,t_0)=\boldsymbol\Psi_{\mathrm{VNB}}(t)\boldsymbol\Phi_K(\tau)$.
Substitution into Eq.~\eqref{eq:vnb-cartesian-stm} gives
\begin{equation}
\boldsymbol\Phi_{\mathrm{VNB}}(t,t_0)=\boldsymbol\Psi_{\mathrm{VNB}}(t)\,\boldsymbol\Phi_K(\tau)\,\boldsymbol\Psi_{\mathrm{VNB}}^{-1}(t_0).
\label{eq:vnb-stm-factorization}
\end{equation}
This factorization establishes that the independent Cartesian and EROE derivations yield the same first-order Keplerian Cartesian propagation.

\section{Numerical Results}
\label{sec:numerical-results}

This section evaluates the formulation numerically in two parts.
Section~\ref{sec:passive-safety-case-study} compares formations sized by the RN-plane separation condition of Ref.~\cite{delurgio2024} and by the BN bound of Proposition~\ref{prop:bn-bound} under a semimajor-axis residual for eccentric chiefs.
Section~\ref{sec:stm-accuracy} verifies the finite-horizon accuracy of the proposed STMs against the YA baseline.

\subsection{Passive-Safety Drift Case Study}
\label{sec:passive-safety-case-study}

This case study compares two E/I-separated formations sized for the same projected-separation target after reserving a margin for semimajor-axis uncertainty at chief eccentricities $e=0.1$, $0.4$, and $0.8$: one by the RN-plane separation condition of Ref.~\cite{delurgio2024}, Eq.~\eqref{eq:legacy-rn-design}, and one by the BN bound of Proposition~\ref{prop:bn-bound}.
The two formations are physically different, and the same semimajor-axis residual is applied to both.

The chief mean angular elements at the initial epoch are $i=\SI{75}{\degree}$, $\Omega=\SI{50}{\degree}$, $\omega=\SI{0}{\degree}$, and $M_0=\SI{0}{\degree}$, and the perigee altitude is fixed at \SI{700}{\kilo\metre}.
Table~\ref{tab:passive-safety-initial-roes} lists the resulting semimajor axes and orbital periods $T_c$.
A spherical keep-out zone (KOZ) of radius $r_{\mathrm{KOZ}}=\SI{500}{\metre}$ is used so that the physical safety metric is invariant to the RTN--VNB rotation.
Both formations target $\Delta a=0$ and are sized for a semimajor-axis error magnitude $|\Delta a_{\mathrm{err}}|=\SI{100}{\metre}$.
Each formation reserves $(1+e)|\Delta a_{\mathrm{err}}|$ for the bounded radial or binormal offset, leaving a projected-separation target $d_{\min}=\SI{700}{\metre}$, \SI{200}{\metre} above the keep-out radius.
The nominal minimum distances in RN- and BN-planes are therefore \SI{810}{\metre}, \SI{840}{\metre}, and \SI{880}{\metre} at the three eccentricities.
A physical residual $\Delta a_{\mathrm{err}}=\SI{100}{\metre}$ is then added to the deputy semimajor axis without redesign, and both cases are propagated for six chief orbits.
This deliberately large residual is a stress parameter chosen to make the secular drift visible over a short horizon.

The Delurgio--D'Amico formation takes $\delta\boldsymbol e^{(D)}\parallel\delta\boldsymbol i^{(D)}$ with equal magnitudes~\cite{delurgio2024}, sized by
\begin{equation}
\frac{a\|\delta\boldsymbol i^{(D)}\|}{1+e}
=d_{\min}+(1+e)|\Delta a_{\mathrm{err}}|.
\label{eq:legacy-rn-design}
\end{equation}
In the proposed formation, $E_\parallel=0$ makes $\delta\boldsymbol e\perp\boldsymbol e$, and the parallel inclination vector is sized through Eq.~\eqref{eq:bn-bound-parallel}, with the residual allowance of Eq.~\eqref{eq:bn-physical-da-residual-bound}, to give the same nominal BN lower bound.
Because $\omega=0$ places $\delta\boldsymbol e$, and hence $\delta\boldsymbol i$, along the negative $y$ axis, $\delta i_x=0$, and the nominal state therefore satisfies the condition of zero differential rate in Eq.~\eqref{eq:j2-stable-design}.
By Proposition~\ref{prop:equivalence}, this nominal state is the drift-free centered safe orbit of Ref.~\cite{peters2021cotangential}, $C_1=C_2=0$ and $\tau_0=0$.
The allowance in Eq.~\eqref{eq:legacy-rn-design} bounds the periodic radial term $\eta^2\Delta a_{\mathrm{err}}/k$ of Eq.~\eqref{eq:rtn-response}, mirroring the near-circular E/I-vector separation practice of reserving clearance for semimajor-axis uncertainty.
In the circular limit, the residual leaves a first-order projected separation of at least \SI{700}{\metre}.
The comparison tests whether this fixed allowance remains sufficient under eccentric Keplerian phase drift.
Both formations share $\delta\boldsymbol i=\delta\boldsymbol i^{(D)}$, and both native longitude coordinates are set to $a\delta\lambda^{(D)}_0=a\delta\lambda_0=\SI{2}{\kilo\metre}$.
The relative eccentricity columns of Table~\ref{tab:passive-safety-initial-roes} are not directly comparable, because by Eq.~\eqref{eq:map-delurgio} the Delurgio--D'Amico formation has $a\delta e=\SI{691}{\metre}$, \SI{376}{\metre}, and \SI{16}{\metre} in the proposed elements, and the factor $1+e$ in the proposed $a\delta e$ follows from the conservative bound of Eq.~\eqref{eq:bn-bound-parallel} rather than from the separation mechanism.

\begin{table*}[hbt!]
\centering
\caption{Chief orbits and nominal formation designs.}
\label{tab:passive-safety-initial-roes}
\begin{tabular}{@{}ccc cccc@{}}
\toprule
& & & \multicolumn{2}{c}{Both formations} & Delurgio--D'Amico & Proposed \\
\cmidrule(lr){4-5}\cmidrule(lr){6-6}\cmidrule(l){7-7}
$e$ & $a$, km & $T_c$, h & Nominal bound, m & $a\delta i$, m & $a\delta e^{(D)}$, m & $a\delta e$, m \\
\midrule
0.1 & 7864.6  & 1.93  & 810 & 891  & 891  & 980.1  \\
0.4 & 11796.9 & 3.54  & 840 & 1176 & 1176 & 1646.4 \\
0.8 & 35390.7 & 18.41 & 880 & 1584 & 1584 & 2851.2 \\
\bottomrule
\end{tabular}
\end{table*}

Each initial ROE vector is converted to a deputy's classical orbital elements.
The chief and deputy are then propagated independently under nonlinear two-body dynamics, converted to Cartesian position at every sample, and differenced in the chief's instantaneous RTN or VNB frame.
All trajectories in the figures and all separations in Table~\ref{tab:passive-safety-minima} are therefore independent of the first-order relative position mappings.

\begin{table*}[hbt!]
\centering
\caption{Six-orbit minimum separations (m) with the \SI{100}{\metre} residual; $r_{\mathrm{KOZ}}=\SI{500}{\metre}$.}
\label{tab:passive-safety-minima}
\begin{tabular}{@{}c cc cc@{}}
\toprule
& \multicolumn{2}{c}{Delurgio--D'Amico} & \multicolumn{2}{c}{Proposed} \\
\cmidrule(lr){2-3}\cmidrule(l){4-5}
$e$ & min.\ RN & Worst 3D & min.\ BN & Worst 3D \\
\midrule
0.1 & 245 & 245 & 778 & 778 \\
0.4 &  26 &  30 & 835 & 835 \\
0.8 & 324 & 325 & 880 & 880 \\
\bottomrule
\end{tabular}
\end{table*}

Table~\ref{tab:passive-safety-minima} reports the minimum distance in the RN and BN planes, which to first order does not depend on the initial longitude, and the worst-case 3D separation over the longitude sweep of Fig.~\ref{fig:passive-safety-lambda-sweeps}.
Without the residual, nonlinear propagation reproduces the nominal minimum RN distances of the Delurgio--D'Amico formation.
The residual activates the secular radial term of Eq.~\eqref{eq:rtn-response}, reducing its RN minimum to as little as \SI{26}{\metre} despite the reserved allowance.
Sizing by the RN-plane separation condition can retain protection over a specified finite horizon by reserving additional clearance for this time-growing term.
The BN minima of the proposed formation remain above the \SI{700}{\metre} guarantee of Eq.~\eqref{eq:bn-physical-da-residual-bound} and match the first-order minima of Eq.~\eqref{eq:bn-minimum} to within \SI{0.5}{\metre}.
For both formations, the worst-case 3D separation nearly equals the projected minimum, so the Delurgio--D'Amico formation enters the keep-out sphere at every eccentricity, whereas the proposed formation never does.

Figures~\ref{fig:passive-safety-e04} and~\ref{fig:passive-safety-e04-motion} show the geometric mechanism at $e=0.4$.
In the projected separation histories of Fig.~\ref{fig:passive-safety-e04-projected}, the RN-plane separation of the Delurgio--D'Amico formation repeatedly drops below the keep-out radius, whereas the BN-plane separation of the proposed formation retains its lower bound.
The corresponding 3D separation in Fig.~\ref{fig:passive-safety-e04-3d} falls to $0.66\,r_{\mathrm{KOZ}}$ for the Delurgio--D'Amico formation but stays above $1.95\,r_{\mathrm{KOZ}}$ for the proposed formation.
In the RN projection of Fig.~\ref{fig:passive-safety-e04-motion-rn}, the zero-$\delta a$ loop of the Delurgio--D'Amico formation remains outside the keep-out circle, but the residual-$\delta a$ trajectory sweeps across it as the phase--energy mode develops a time-growing radial component.
For the proposed formation, the BN projection of Fig.~\ref{fig:passive-safety-e04-motion-bn} shows that the residual adds the bounded binormal offset $\eta^2\Delta a_{\mathrm{err}}/q$ of Sec.~\ref{sec:passive-safety}.
As Fig.~\ref{fig:passive-safety-e04-motion-vb} illustrates, the accumulated phase drift remains confined to $V$ and cannot create secular BN growth.
This behavior is predicted by Eq.~\eqref{eq:vnb-secular} and by the fixed margin in Eq.~\eqref{eq:bn-physical-da-residual-bound}.

\begin{figure}[hbt!]
\centering
\begin{subfigure}[t]{0.49\textwidth}
\centering
\includegraphics[width=\linewidth]{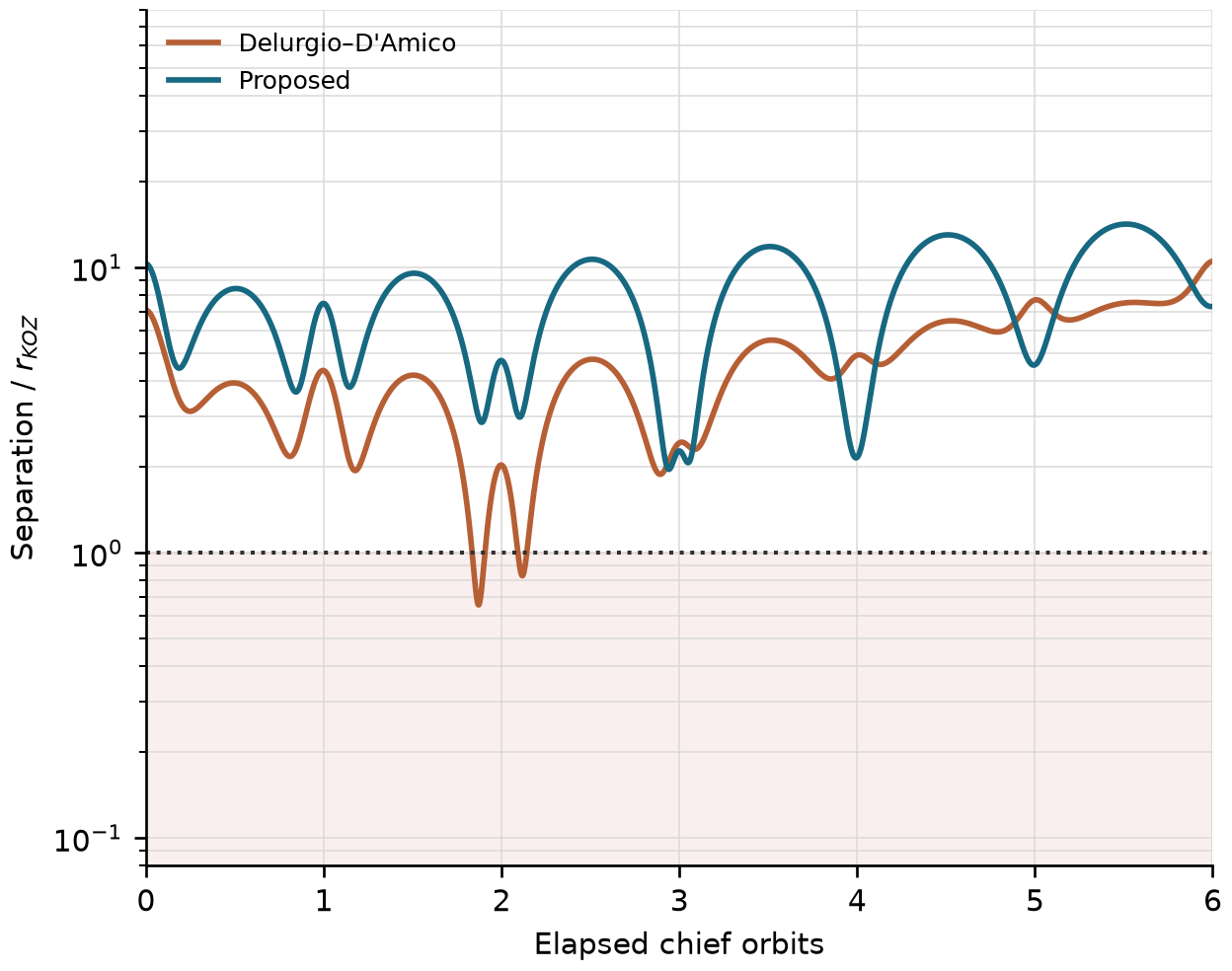}
\caption{Full three-dimensional separation.}
\label{fig:passive-safety-e04-3d}
\end{subfigure}\hfill
\begin{subfigure}[t]{0.49\textwidth}
\centering
\includegraphics[width=\linewidth]{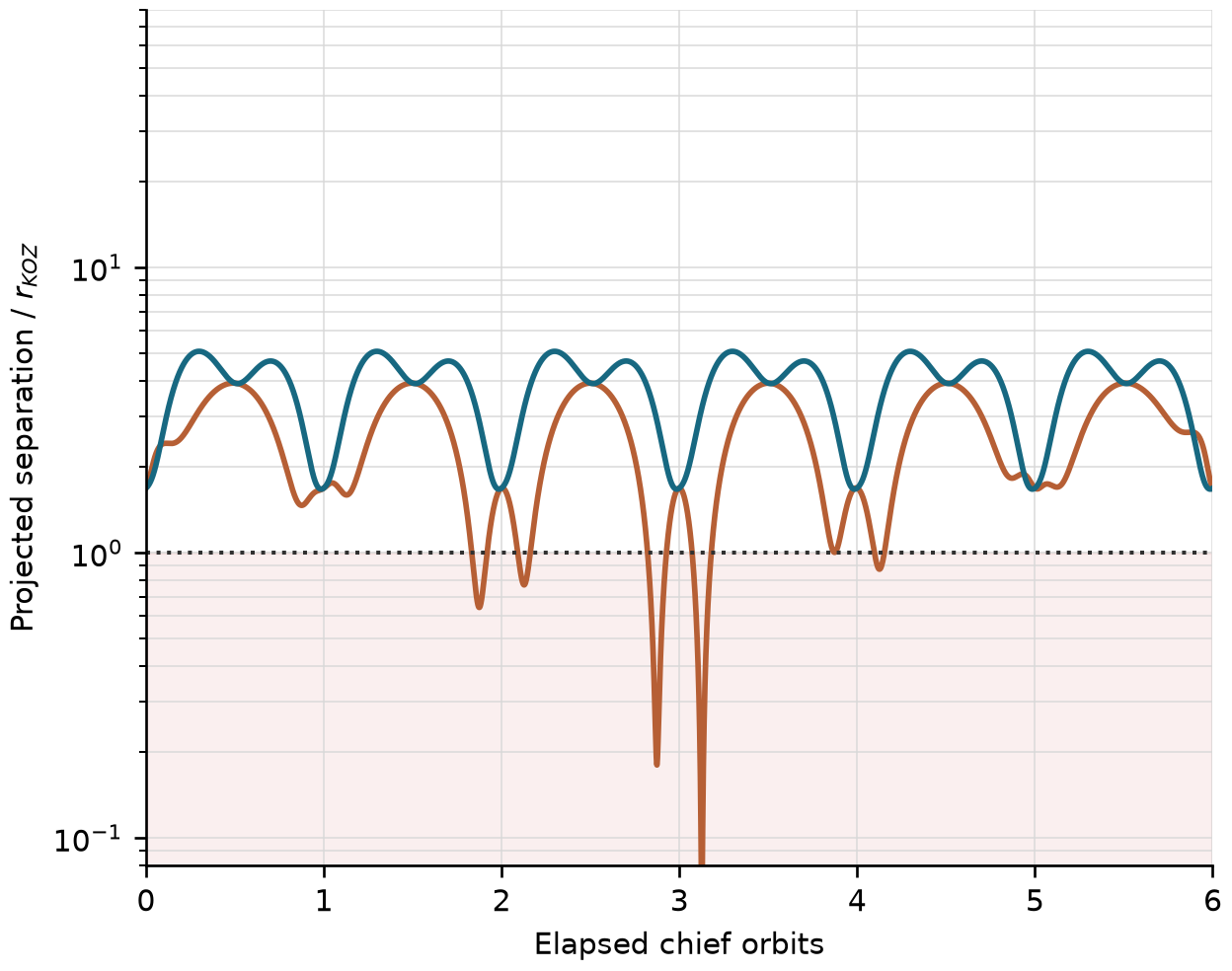}
\caption{RN-plane separation (Delurgio--D'Amico) and BN-plane separation (proposed).}
\label{fig:passive-safety-e04-projected}
\end{subfigure}
\caption{Nonlinear separation histories for $e=0.4$ with the \SI{100}{\metre} residual and $a\delta\lambda^{(D)}_0=a\delta\lambda_0=\SI{2}{\kilo\metre}$, normalized by $r_{\mathrm{KOZ}}$. The shaded region marks the keep-out zone.}
\label{fig:passive-safety-e04}
\end{figure}
\begin{figure}[hbt!]
\centering
\begin{subfigure}[t]{0.49\textwidth}
\centering
\includegraphics[width=\linewidth]{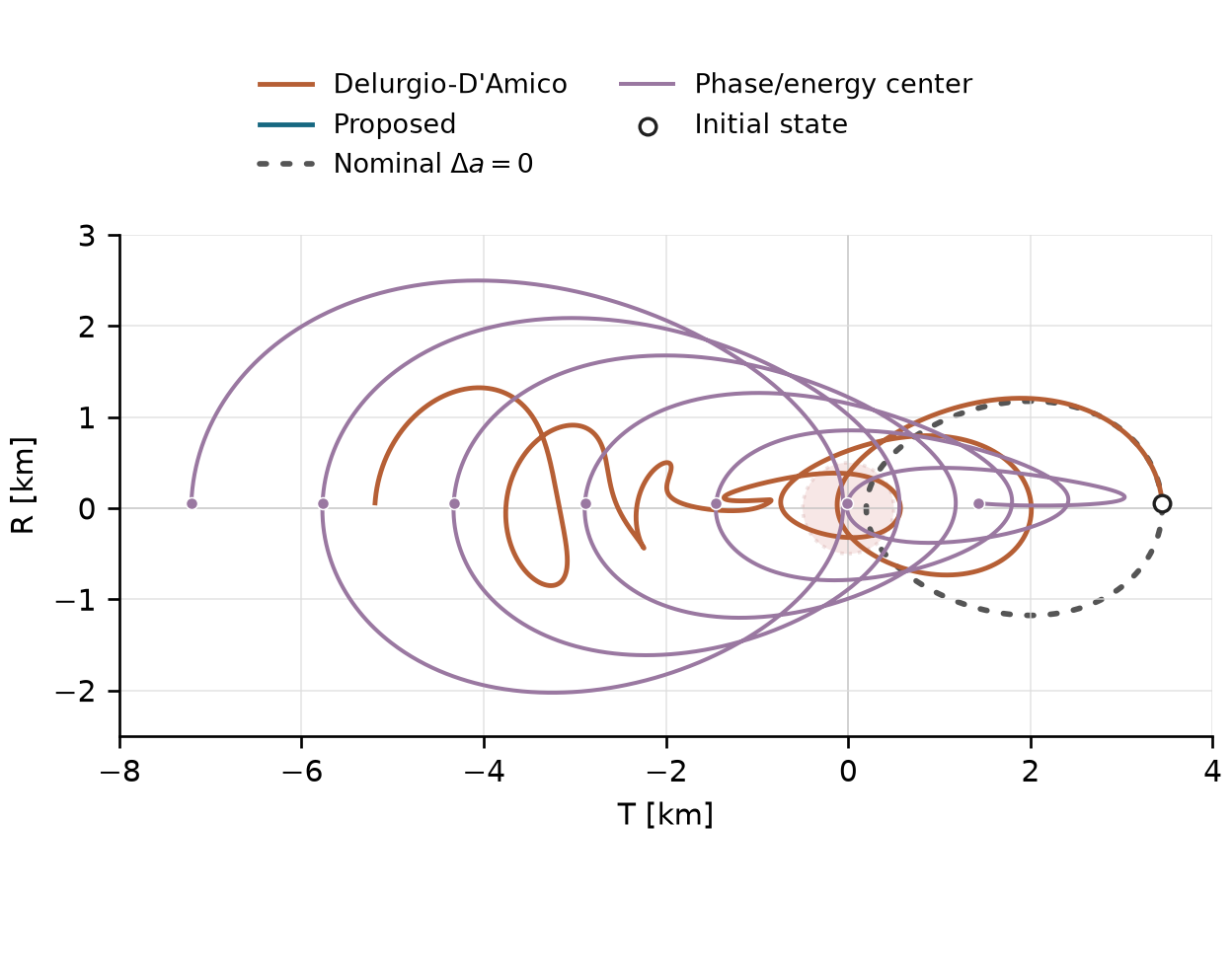}
\caption{Delurgio--D'Amico formation, RT projection.}
\label{fig:passive-safety-e04-motion-rt}
\end{subfigure}\hfill
\begin{subfigure}[t]{0.49\textwidth}
\centering
\includegraphics[width=\linewidth]{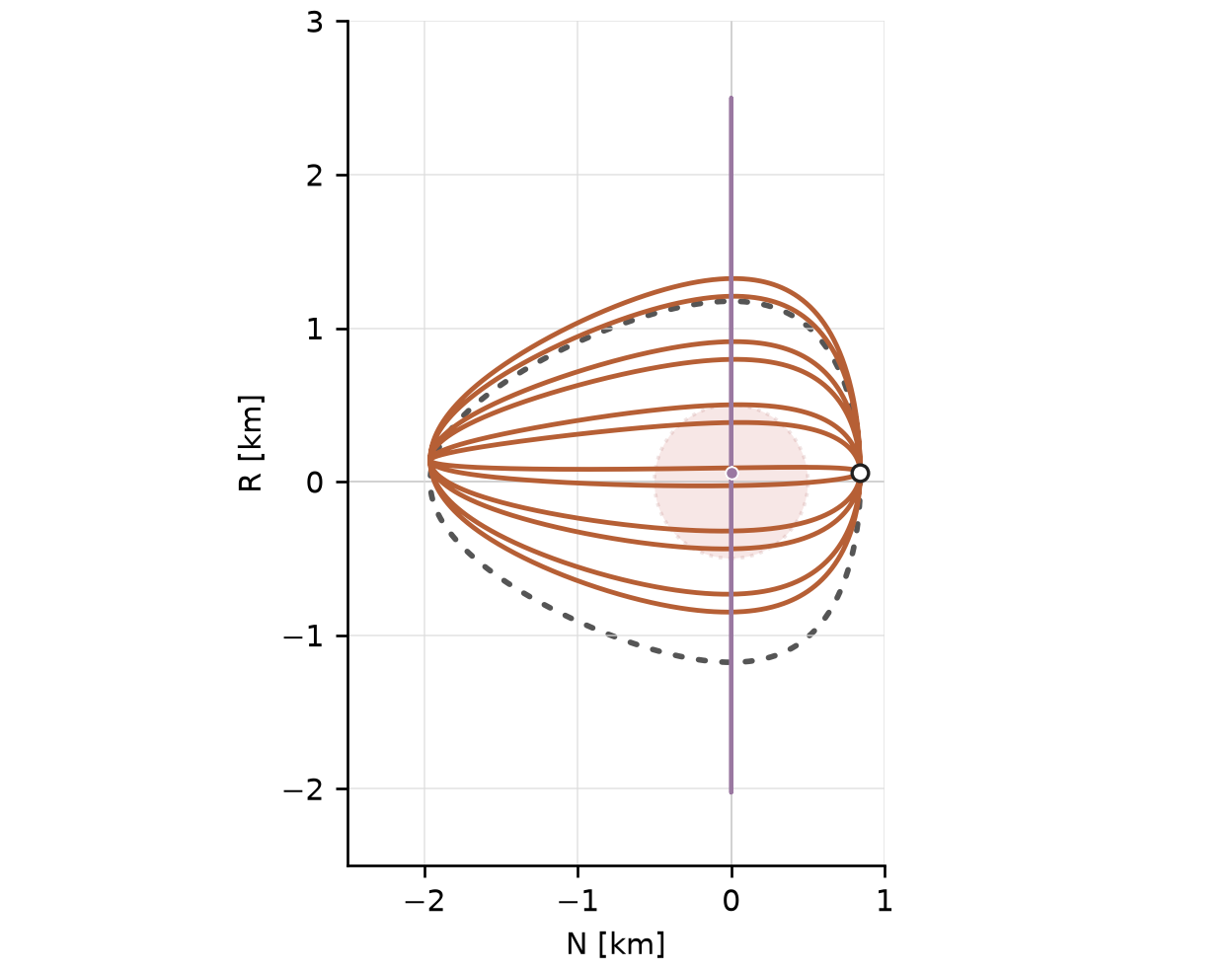}
\caption{Delurgio--D'Amico formation, RN projection.}
\label{fig:passive-safety-e04-motion-rn}
\end{subfigure}
\par\medskip
\begin{subfigure}[t]{0.49\textwidth}
\centering
\includegraphics[width=\linewidth]{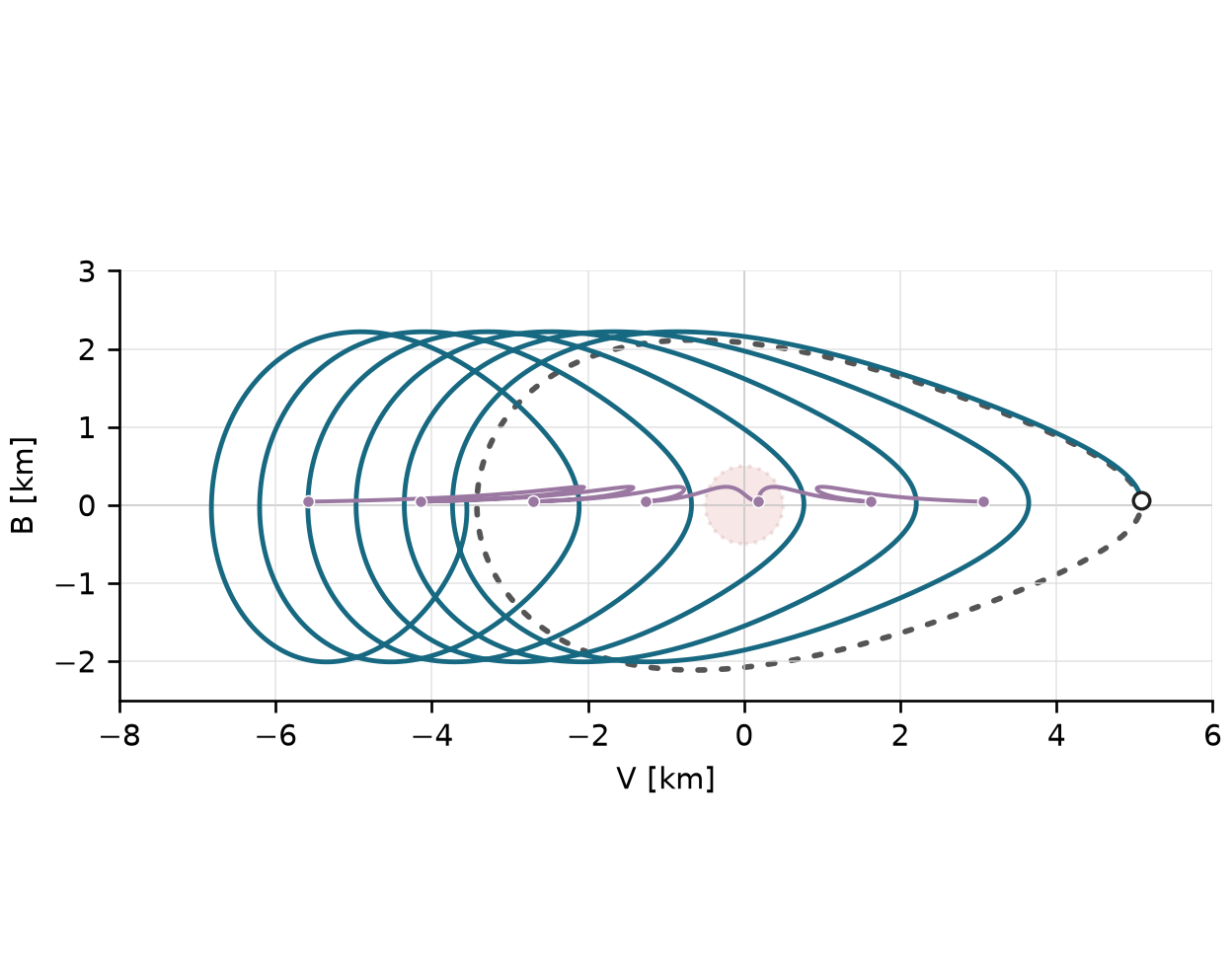}
\caption{Proposed formation, VB projection.}
\label{fig:passive-safety-e04-motion-vb}
\end{subfigure}\hfill
\begin{subfigure}[t]{0.49\textwidth}
\centering
\includegraphics[width=\linewidth]{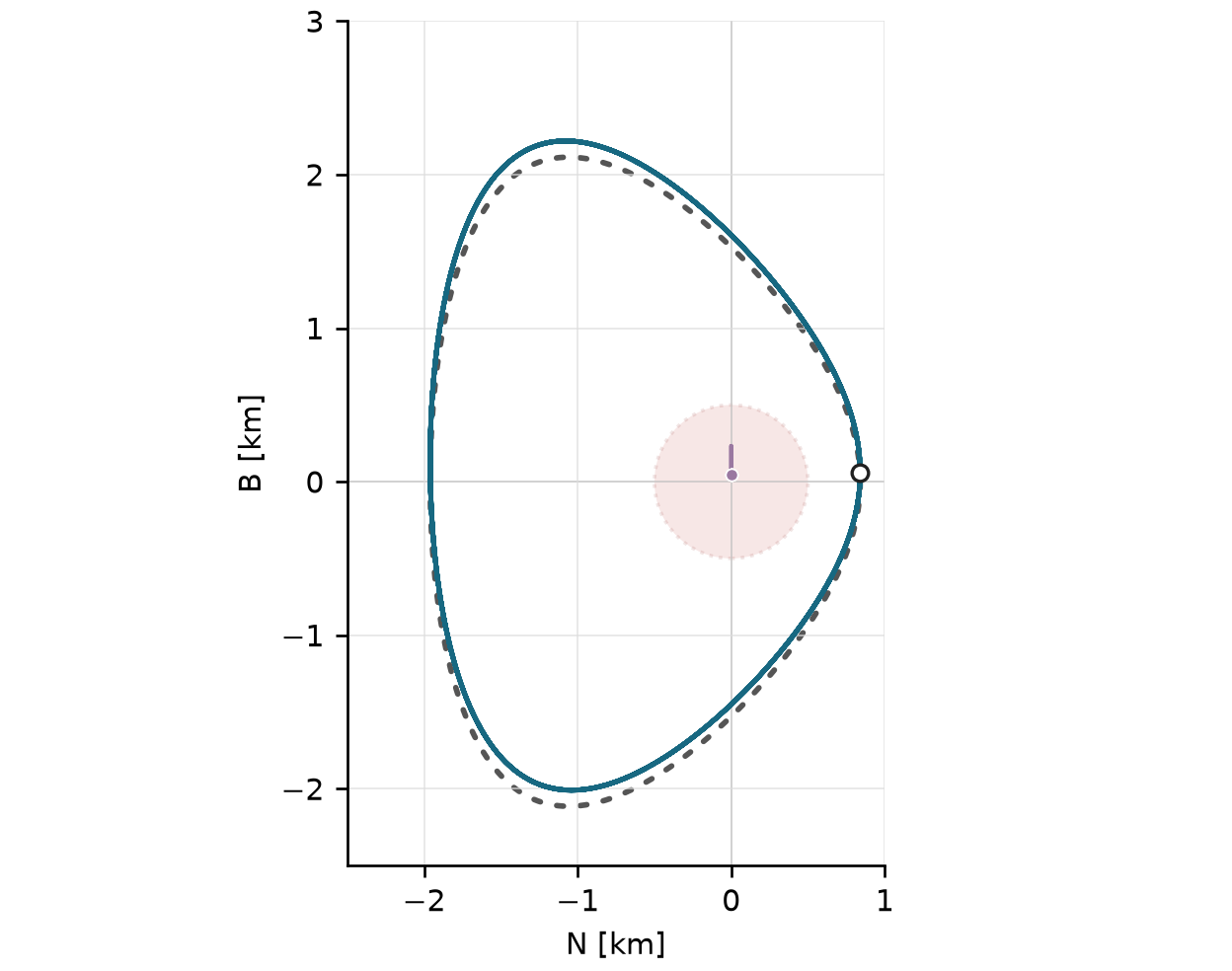}
\caption{Proposed formation, BN projection.}
\label{fig:passive-safety-e04-motion-bn}
\end{subfigure}
\caption{Relative-motion projections for the case of Fig.~\ref{fig:passive-safety-e04}. Solid curves include the \SI{100}{\metre} residual, and dotted curves show the nominal $\Delta a=0$ design. Purple curves are the first-order phase/energy centers, with dots at integer chief orbits. Shaded disks mark the keep-out zone, and open circles the initial state.}
\label{fig:passive-safety-e04-motion}
\end{figure}

Figure~\ref{fig:passive-safety-lambda-sweeps} tests sensitivity to the initial native longitude coordinate at all three eccentricities.
Each panel evaluates
$d_{\min}^{\mathrm{3D}}(L_0)=\min_{0\le t\le6T_c}\|\boldsymbol\rho(t)\|$
at \SI{100}{\metre} increments over $-\SI{5}{\kilo\metre}\le L_0\le\SI{10}{\kilo\metre}$, where $L_0=a\delta\lambda^{(D)}_0$ for the Delurgio--D'Amico formation and $L_0=a\delta\lambda_0$ for the proposed formation.
The upper limit exceeds the six-orbit longitude drift of both formations, including the larger drift of the Delurgio--D'Amico formation at $e=0.8$.
All nonlongitude elements remain fixed at the values of Table~\ref{tab:passive-safety-initial-roes}.
Because the two longitude coordinates have different definitions, equal abscissa values do not correspond to identical physical initial states.
The Delurgio--D'Amico formation enters the keep-out sphere over longitude intervals at every eccentricity, whereas the floor of the proposed formation is the bounded BN margin at every longitude.
The minima of these curves are the worst-case entries of Table~\ref{tab:passive-safety-minima}.

\begin{figure}[hbt!]
\centering
\begin{subfigure}[t]{0.8\textwidth}
\centering
\includegraphics[width=\linewidth]{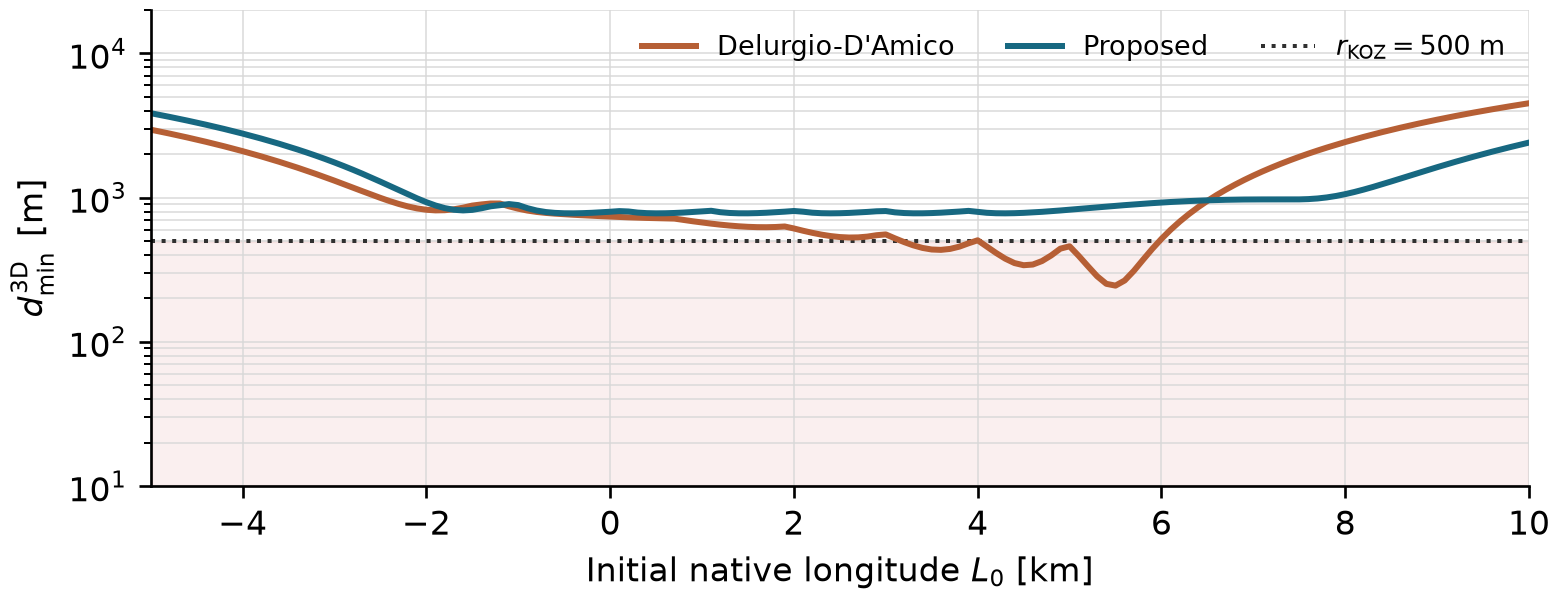}
\caption{$e=0.1$.}
\label{fig:passive-safety-lambda-e01}
\end{subfigure}
\par\medskip
\begin{subfigure}[t]{0.8\textwidth}
\centering
\includegraphics[width=\linewidth]{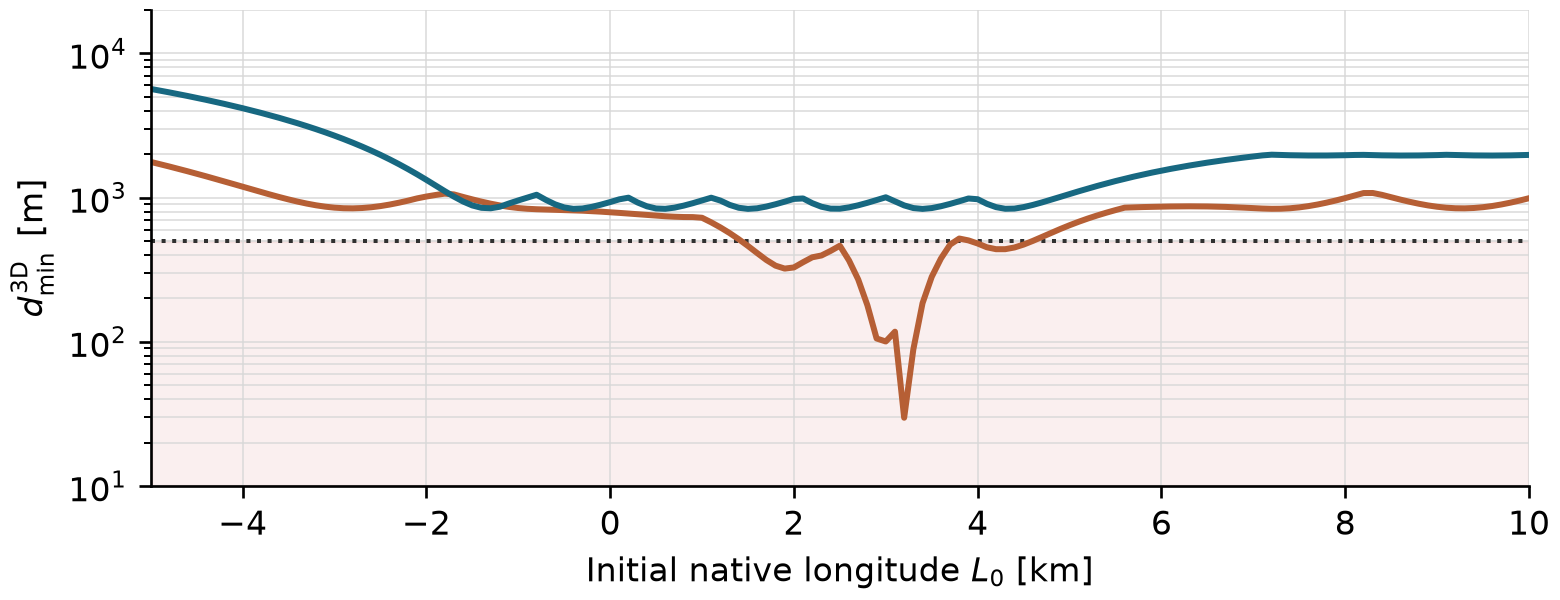}
\caption{$e=0.4$.}
\label{fig:passive-safety-lambda-e04}
\end{subfigure}
\par\medskip
\begin{subfigure}[t]{0.8\textwidth}
\centering
\includegraphics[width=\linewidth]{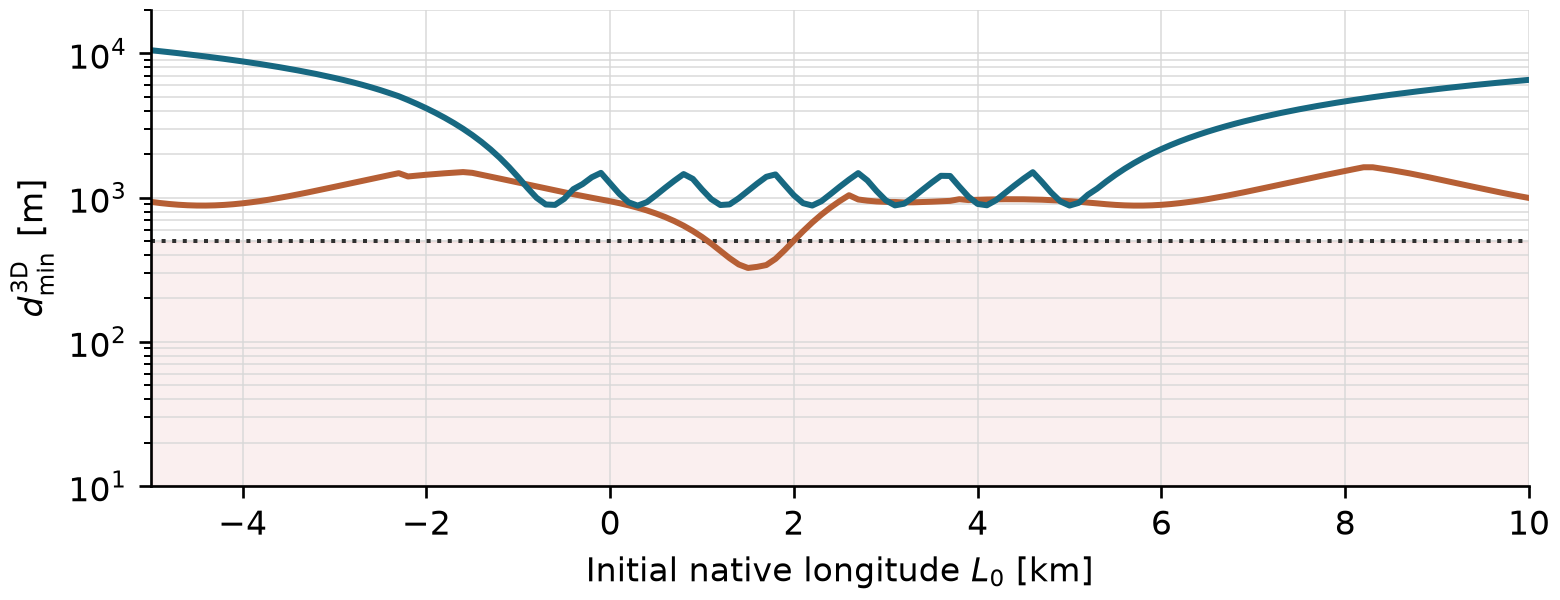}
\caption{$e=0.8$.}
\label{fig:passive-safety-lambda-e08}
\end{subfigure}
\caption{Six-orbit minimum three-dimensional separation versus the initial native longitude coordinate. Brown and blue curves show the Delurgio--D'Amico and proposed formations, respectively. The dotted line marks the \SI{500}{\metre} keep-out radius, with violations shaded below.}
\label{fig:passive-safety-lambda-sweeps}
\end{figure}

\subsection{Finite-Horizon STM Accuracy}
\label{sec:stm-accuracy}

This comparison does not seek higher first-order accuracy than the YA solution.
It tests whether the proposed element STMs, reconstructed through the VNB position map, reproduce the same first-order relative dynamics as the YA baseline.
Figure~\ref{fig:stm-nonlinear-error} reports the maximum position error over ten chief orbits as a function of chief eccentricity and initial longitude offset.
This error is defined by
\begin{equation}
 \epsilon_{\max}=\max_{0\leq t\leq 10T_c}
 \left\|\boldsymbol{\rho}_{\mathrm{STM}}(t)-
 \boldsymbol{\rho}_{\mathrm{ref}}(t)\right\|.
\end{equation}
The baseline deputy is initialized at $e=0.4$ from the proposed elements with $\Delta a=\SI{100}{\metre}$ and $a\delta\lambda_0=\SI{1}{\kilo\metre}$, and the remaining elements impose $\delta p=0$ and a \SI{700}{\metre} design clearance through Eq.~\eqref{eq:bn-bound}.
The eccentricity sweep varies $0.1\leq e\leq0.8$ at the fixed \SI{700}{\kilo\metre} perigee altitude of Sec.~\ref{sec:passive-safety-case-study}, holding the chief angular elements of that section and the baseline dimensional relative-element vector fixed.
The longitude sweep fixes $e=0.4$ and varies $0\leq a\delta\lambda_0\leq\SI{10}{\kilo\metre}$ with all other elements fixed.
At every sweep point, both formulations use the same physical chief and deputy.

For the reference, the chief and deputy mean elements are propagated independently with the secular $J_2$ rates of Eq.~\eqref{eq:j2-chief-rates}, mapped to Cartesian position, and differenced.
This reference is nonlinear in the finite chief--deputy element differences but shares the averaged-$J_2$ scope of the analytical STMs.
It excludes short-period terms and mean-to-osculating reconstruction.
Each STM prediction reconstructs position from the propagated ROE with its first-order map, Eq.~\eqref{eq:rtn-response} in RTN or Eq.~\eqref{eq:position} in VNB, and is rotated into inertial axes before the error is evaluated, so the comparison does not depend on the rotating-frame coordinates.
The YA STM curves use the Delurgio--D'Amico normalization of the YA integration constants~\cite{delurgio2024}.
The Sullivan--D'Amico reference-epoch EROE are first-order equivalent to the same constants~\cite{sullivan2017nkf}.
With consistent initialization and linear RTN reconstruction, their Keplerian prediction reproduces the YA response and is therefore not shown.

The two secular-$J_2$ STM error curves in Fig.~\ref{fig:stm-nonlinear-error} are nearly coincident in both sweeps.
Across the eccentricity sweep, their ten-orbit maxima grow from approximately \SI{5}{\metre} at low eccentricity to \SI{25}{\metre} at $e=0.8$.
The longitude sweep is nonmonotonic, because varying only $\delta\lambda_0$ allows finite-offset truncation terms associated with the fixed remaining elements to partially cancel.
Both formulations reach a minimum of approximately \SI{3}{\metre} near $a\delta\lambda_0=\SI{4}{\kilo\metre}$.
The close agreement is consistent with the two descriptions representing the same first-order relative dynamics, and their small residual difference reflects finite-offset truncation in the distinct first-order coordinate maps.
The two Keplerian STMs closely follow nonlinear Keplerian propagation, and all three have errors one to two orders of magnitude larger than those of the $J_2$ STMs because they omit differential $J_2$ drift.

\begin{figure}[hbt!]
\centering
\begin{subfigure}[t]{0.49\linewidth}
\centering
\includegraphics[width=\linewidth]{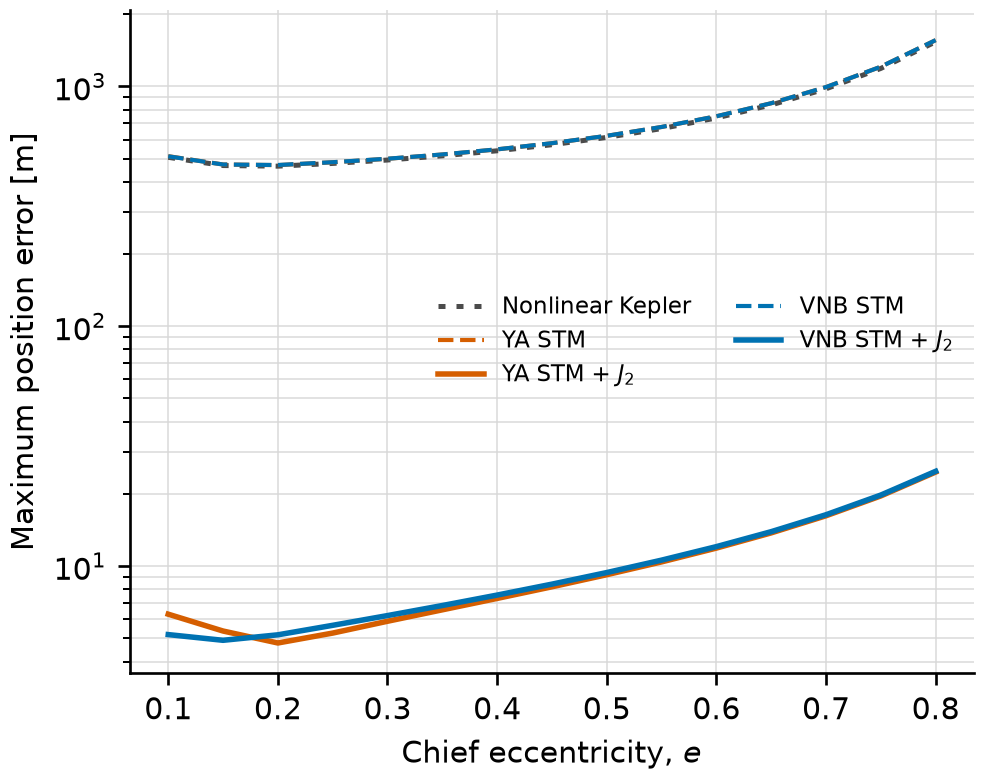}
\caption{Chief-eccentricity sweep, $a\delta\lambda_0=\SI{1}{\kilo\metre}$.}
\label{fig:stm-error-eccentricity}
\end{subfigure}\hfill
\begin{subfigure}[t]{0.49\linewidth}
\centering
\includegraphics[width=\linewidth]{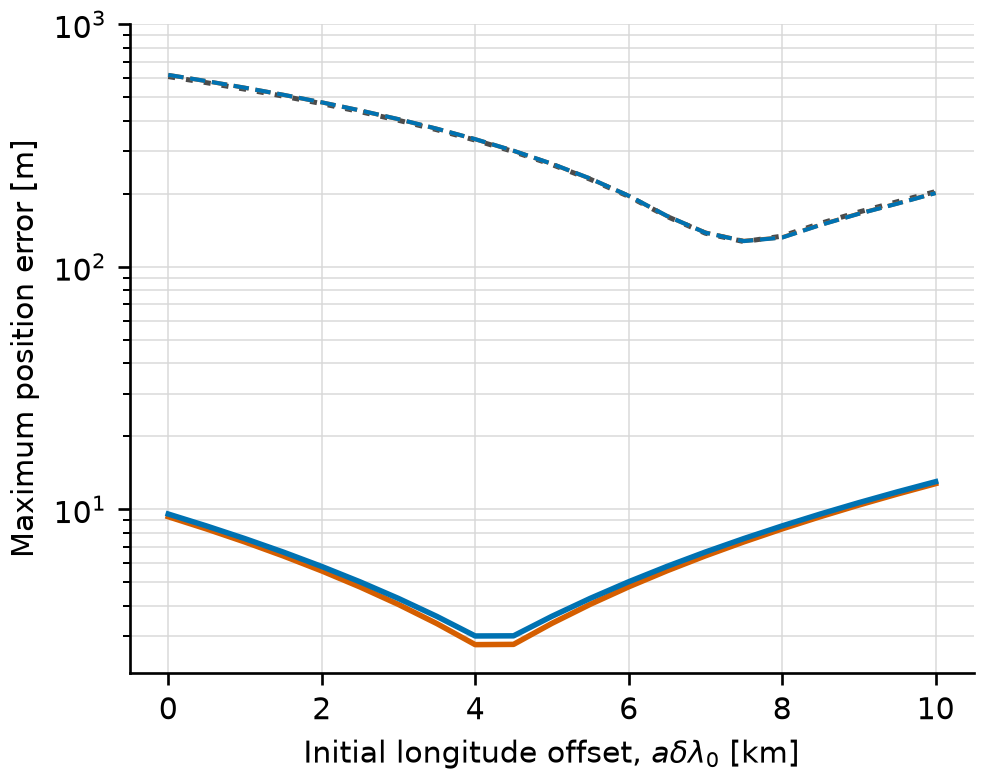}
\caption{Initial-longitude sweep, $e=0.4$.}
\label{fig:stm-error-longitude}
\end{subfigure}
\caption{Maximum position error over ten chief orbits relative to the nonlinear secular-mean-$J_2$ reference.}
\label{fig:stm-nonlinear-error}
\end{figure}

\section{Conclusion}
\label{sec:conclusions}

This paper presented a unified ROE framework for formation flying in eccentric orbits, built on a single qnsROE-like element set whose geometry is resolved in the VNB frame.
Because a first-order phase offset displaces the deputy along the chief velocity, this frame confines the secular phase drift to the velocity axis and leaves the BN geometry phase independent.
The proposed EROE are regular at $e=0$, reduce to the qnsROE in the circular limit, and map explicitly to other ROE sets and to the integration constants of the linearized relative dynamics, coinciding with the latter under a physical normalization of the Cartesian VNB solution.
With these elements, a closed-form BN-plane separation bound holds for arbitrary E/I-vector orientation, with a fixed clearance allowance for a semimajor-axis residual.
Closed-form Keplerian and secular-$J_2$ STMs and a decoupled impulsive control input matrix support propagation and control, and each reduces to its qnsROE counterpart at $e=0$.
For eccentric chiefs, numerical results confirmed the BN bound with the residual margin and showed formations sized by the RN-plane separation condition entering the keep-out zone.
Extending the Keplerian certificate to perturbed long-horizon motion and uncertainty dispersions, and building closed-form impulsive control on the decoupled input matrix, are natural directions for future work.

\appendix
\titleformat{\section}{\large\bfseries\singlespacing\centering}{Appendix~\thesection:\space}{0pt}{#1}[]

\section{Regularization at Zero Eccentricity}
\label{app:regularization}

At $e=0$, where $\omega$ is undefined, the classical-element form of Eq.~\eqref{eq:roe-definition} is singular in $\delta\lambda$ and both components of $\delta\boldsymbol e$.
An equivalent definition that is regular at $e=0$ uses the node-corrected eccentricity differences
\begin{equation}
\begin{aligned}
\Delta\epsilon_x&=\Delta e_x-e_y\cos i\,\Delta\Omega,\\
\Delta\epsilon_y&=\Delta e_y+e_x\cos i\,\Delta\Omega.
\end{aligned}
\end{equation}
In terms of these differences, the regular definition is
\begin{equation}
\begin{aligned}
\delta\lambda&=\delta\lambda^{(\mathrm{qns})}+
F(e^2)(e_x\Delta\epsilon_y-e_y\Delta\epsilon_x),\\
\delta e_x&=(\eta^2+2e_x^2)\Delta\epsilon_x
+2e_xe_y\Delta\epsilon_y-\eta^2 \dfrac{\Delta a}{a} e_x,\\
\delta e_y&=2e_x e_y\Delta\epsilon_x
+(\eta^2+2e_y^2)\Delta\epsilon_y-\eta^2 \dfrac{\Delta a}{a} e_y.
\end{aligned}
\end{equation}
The scalar factor $F$ is defined by
\begin{equation}
F(x)=-\frac{5-2x+x^2}
{(1+x)((1-x)^{3/2}+1+x)},\qquad F(0)=-\frac52.
\label{eq:phase-factor}
\end{equation}
The factor follows from
\begin{equation}
F(x)=\frac{1}{x}\left[\frac{(1-x)^{3/2}}{1+x}-1\right]
=\frac{(1-x)^{3/2}-(1+x)}{x(1+x)}.
\end{equation}
Rationalizing the numerator with $(1-x)^3-(1+x)^2=-x(5-2x+x^2)$ gives Eq.~\eqref{eq:phase-factor} and exposes the removable singularity.
Because $e_x\Delta\epsilon_y-e_y\Delta\epsilon_x=e^2 (\Delta\omega+\cos i\,\Delta\Omega)$ and $1+e^2F(e^2)=\eta^3/\chi$, the longitude agrees with Eq.~\eqref{eq:roe-definition}.
The transformation is regular for $0\le e<1$ and retains only the equatorial singularity of the QNS elements.
As $e\to0$, $F$ remains finite and $e_x,e_y\to0$, so the proposed EROE reduce to the qnsROE of Table~\ref{tab:roe_representation_comparison}.

The inverse is equally explicit.
The matrix $\eta^2\boldsymbol I_2+2\boldsymbol e\boldsymbol e^{\top}$ has eigenvalues $\chi$ along $\boldsymbol e$ and $\eta^2$ normal to it, so the inverse relations are
\begin{equation}
\Delta\boldsymbol\epsilon=\frac{1}{\eta^2}\left(\boldsymbol I_2-\frac{2}{\chi}\boldsymbol e\boldsymbol e^{\top}\right)
\left(\delta\boldsymbol e+\eta^2\delta a\,\boldsymbol e\right),
\qquad
\delta\lambda^{(\mathrm{qns})}=\delta\lambda-F(e^2)(e_x\Delta\epsilon_y-e_y\Delta\epsilon_x),
\label{eq:roe-inverse}
\end{equation}
where $\Delta\boldsymbol\epsilon=[\Delta\epsilon_x,\Delta\epsilon_y]^{\top}$.
The remaining element differences are $\Delta a/a=\delta a$, $\Delta i=\delta i_x/\eta^2$, and $\sin i\,\Delta\Omega=\delta i_y/\eta^2$.
The QNS eccentricity differences then follow as $\Delta e_x=\Delta\epsilon_x+e_y\cos i\,\Delta\Omega$ and $\Delta e_y=\Delta\epsilon_y-e_x\cos i\,\Delta\Omega$, so the conversion between the proposed EROE and the QNS element differences is regular in both directions for $0\le e<1$.

\section*{Acknowledgments}

The authors gratefully acknowledge the researchers who pioneered the concept of relative eccentricity/inclination-vector separation for geostationary satellite colocation and subsequently generalized it to near-circular formation flying.
Yuji Takubo thanks the Ezoe Memorial Recruit Foundation for the financial support, and Shane Lowe for insightful discussions.
Artificial intelligence tools were used to assist in the development of the proposed EROE framework, including algebraic derivations, code generation, and manuscript preparation.
The authors are fully responsible to what is presented in this paper.

\bibliography{reference}

@article{damico2012prisma,
  title   = {Spaceborne Autonomous Formation-Flying Experiment on the {PRISMA} Mission},
  author  = {D'Amico, Simone and Ardaens, Jean-Sebastien and Larsson, Robin},
  journal = {Journal of Guidance, Control, and Dynamics},
  volume  = {35},
  number  = {3},
  pages   = {834--850},
  year    = {2012},
  doi     = {10.2514/1.55638}
}

@article{krieger2007tandemx,
  title   = {{TanDEM-X}: A Satellite Formation for High-Resolution {SAR} Interferometry},
  author  = {Krieger, Gerhard and Moreira, Alberto and Fiedler, Hauke and Hajnsek, Irena and Werner, Marian and Younis, Marwan and Zink, Manfred},
  journal = {IEEE Transactions on Geoscience and Remote Sensing},
  volume  = {45},
  number  = {11},
  pages   = {3317--3341},
  year    = {2007},
  doi     = {10.1109/TGRS.2007.900693}
}

@article{montenbruck2006grace,
  title   = {{E/I}-Vector Separation for Safe Switching of the {GRACE} Formation},
  author  = {Montenbruck, Oliver and Kirschner, Michael and D'Amico, Simone and Bettadpur, Srinivas},
  journal = {Aerospace Science and Technology},
  volume  = {10},
  number  = {7},
  pages   = {628--635},
  year    = {2006},
  doi     = {10.1016/j.ast.2006.04.001}
}

@article{kruger2023starfox,
  title   = {Starling Formation-Flying Optical Experiment ({StarFOX}): System Design and Preflight Verification},
  author  = {Kruger, Justin and Koenig, Adam W. and D'Amico, Simone},
  journal = {Journal of Spacecraft and Rockets},
  volume  = {60},
  number  = {6},
  pages   = {1755--1777},
  year    = {2023},
  doi     = {10.2514/1.A35598}
}

@article{koenig2023visors,
  title   = {Formation Flying Orbit and Control Concept for the Virtual Super Optics Reconfigurable Swarm Mission},
  author  = {Koenig, Adam W. and D'Amico, Simone and Lightsey, E. Glenn},
  journal = {Journal of Guidance, Control, and Dynamics},
  volume  = {46},
  number  = {9},
  pages   = {1657--1670},
  year    = {2023},
  doi     = {10.2514/1.G007334}
}

@inproceedings{agarwal2021swarmex,
  title     = {Coordinating Development of the {SWARM-EX} {CubeSat} Swarm Across Multiple Institutions},
  author    = {Agarwal, Rahul and Oh, Bryce and Fitzpatrick, Daniel and Buynovskiy, Andrew and Lowe, Samuel and Lisy, Charles and Kriezis, Alexander and Lan, Brandon and Lee, Zachary and Thomas, Alan and Wallace, Bryan and Costantino, Elise and Miner, Grace and Thayer, Jeffrey and D'Amico, Simone and Lemmer, Kristina and Lohmeyer, Whitney and Palo, Scott},
  booktitle = {35th Annual Small Satellite Conference},
  address   = {Logan, UT},
  year      = {2021}
}

@article{burch2016mms,
  title   = {Magnetospheric Multiscale Overview and Science Objectives},
  author  = {Burch, J. L. and Moore, T. E. and Torbert, R. B. and Giles, B. L.},
  journal = {Space Science Reviews},
  volume  = {199},
  number  = {1--4},
  pages   = {5--21},
  year    = {2016},
  doi     = {10.1007/s11214-015-0164-9}
}

@inproceedings{plice2019helioswarm,
  title     = {{HelioSwarm}: Swarm Mission Design in High Altitude Orbit for Heliophysics},
  author    = {Plice, Laura and Dono Perez, Andres and West, Steven},
  booktitle = {AAS/AIAA Astrodynamics Specialist Conference},
  note      = {Paper AAS 19-831},
  address   = {Portland, ME},
  year      = {2019}
}

@inproceedings{lillis2020escapade,
  title     = {{ESCAPADE}: Coordinated Multipoint Measurements of {Mars}' Near-Space Plasma Environment},
  author    = {Lillis, Robert and Curry, Shannon and Russell, Christopher and Luhmann, Janet and Barjatya, Aroh and Larson, Davin and Modolo, Ronan and Livi, Roberto and Whittlesey, Phyllis and Harada, Yuki and Fowler, Christopher and Xu, Shaosui and Brain, David and Withers, Paul and Thiemann, Edward},
  booktitle = {EGU General Assembly},
  pages     = {{EGU2020-6354}},
  year      = {2020},
  doi       = {10.5194/egusphere-egu2020-6354}
}

@inproceedings{kolmas2016mdot,
  title     = {System Design of a Miniaturized Distributed Occulter/Telescope for Direct Imaging of Star Vicinity},
  author    = {Kolmas, Jan and Banazadeh, Payam and Koenig, Adam W. and Macintosh, Bruce and D'Amico, Simone},
  booktitle = {2016 IEEE Aerospace Conference},
  pages     = {1--11},
  year      = {2016},
  doi       = {10.1109/AERO.2016.7500783}
}

@article{llorente2013proba3,
  title   = {{PROBA-3}: Precise Formation Flying Demonstration Mission},
  author  = {Llorente, J. S. and Agenjo, A. and Carrascosa, C. and de Negueruela, C. and Mestreau-Garreau, A. and Cropp, A. and Santovincenzo, A.},
  journal = {Acta Astronautica},
  volume  = {82},
  number  = {1},
  pages   = {38--46},
  year    = {2013},
  doi     = {10.1016/j.actaastro.2012.05.029}
}

@article{clohessy1960,
  title   = {Terminal Guidance System for Satellite Rendezvous},
  author  = {Clohessy, W. H. and Wiltshire, R. S.},
  journal = {Journal of the Aerospace Sciences},
  volume  = {27},
  number  = {9},
  pages   = {653--658},
  year    = {1960},
  doi     = {10.2514/8.8704}
}

@article{tschauner1965,
  title   = {Rendezvous zu einem in elliptischer Bahn umlaufenden Ziel},
  author  = {Tschauner, J. and Hempel, P.},
  journal = {Astronautica Acta},
  volume  = {11},
  number  = {2},
  pages   = {104--109},
  year    = {1965}
}

@phdthesis{guffanti2022optimal,
  title   = {{Optimal Passively-Safe Control of Multi-Agent Motion with Application to Distributed Space Systems}},
  author  = {Guffanti, Tommaso},
  type    = {{Ph.D. Thesis}},
  school  = {Stanford Univ.},
  address = {Stanford, CA},
  year    = {2022},
  url     = {https://purl.stanford.edu/gh147jp5825}
}

@article{yamanaka2002,
  title   = {New State Transition Matrix for Relative Motion on an Arbitrary Elliptical Orbit},
  author  = {Yamanaka, Koji and Ankersen, Finn},
  journal = {Journal of Guidance, Control, and Dynamics},
  volume  = {25},
  number  = {1},
  pages   = {60--66},
  year    = {2002},
  doi     = {10.2514/2.4875}
}

@article{sullivan2017survey,
  title   = {Comprehensive Survey and Assessment of Spacecraft Relative Motion Dynamics Models},
  author  = {Sullivan, Joshua and Grimberg, Sebastian and D'Amico, Simone},
  journal = {Journal of Guidance, Control, and Dynamics},
  volume  = {40},
  number  = {8},
  pages   = {1837--1859},
  year    = {2017},
  doi     = {10.2514/1.G002309}
}

@book{alfriend2010,
  title     = {Spacecraft Formation Flying: Dynamics, Control and Navigation},
  author    = {Alfriend, Kyle T. and Vadali, Srinivas R. and Gurfil, Pini and How, Jonathan P. and Breger, Louis S.},
  publisher = {Elsevier},
  address   = {Oxford, UK},
  year      = {2010}
}

@article{casotto1993,
  title   = {Position and Velocity Perturbations in the Orbital Frame in Terms of Classical Element Perturbations},
  author  = {Casotto, Stefano},
  journal = {Celestial Mechanics and Dynamical Astronomy},
  volume  = {55},
  number  = {3},
  pages   = {209--221},
  year    = {1993},
  doi     = {10.1007/BF00692510}
}

@article{schaub2004,
  title   = {Relative Orbit Geometry Through Classical Orbit Element Differences},
  author  = {Schaub, Hanspeter},
  journal = {Journal of Guidance, Control, and Dynamics},
  volume  = {27},
  number  = {5},
  pages   = {839--848},
  year    = {2004},
  doi     = {10.2514/1.12595}
}

@article{gim2003,
  title   = {State Transition Matrix of Relative Motion for the Perturbed Noncircular Reference Orbit},
  author  = {Gim, Dong-Woo and Alfriend, Kyle T.},
  journal = {Journal of Guidance, Control, and Dynamics},
  volume  = {26},
  number  = {6},
  pages   = {956--971},
  year    = {2003},
  doi     = {10.2514/2.6924}
}

@article{damico2006ei,
  title   = {Proximity Operations of Formation-Flying Spacecraft Using an Eccentricity/Inclination Vector Separation},
  author  = {D'Amico, Simone and Montenbruck, Oliver},
  journal = {Journal of Guidance, Control, and Dynamics},
  volume  = {29},
  number  = {3},
  pages   = {554--563},
  year    = {2006},
  doi     = {10.2514/1.15114}
}

@phdthesis{damico2010thesis,
  title  = {Autonomous Formation Flying in Low {Earth} Orbit},
  author = {D'Amico, Simone},
  school = {Delft University of Technology},
  address = {Delft, The Netherlands},
  year   = {2010}
}

@article{sullivan2017nkf,
  title   = {Nonlinear {Kalman} Filtering for Improved Angles-Only Navigation Using Relative Orbital Elements},
  author  = {Sullivan, Joshua and D'Amico, Simone},
  journal = {Journal of Guidance, Control, and Dynamics},
  volume  = {40},
  number  = {9},
  pages   = {2183--2200},
  year    = {2017},
  doi     = {10.2514/1.G002719}
}

@phdthesis{willis2023thesis,
  title  = {Analytical Theory of Satellite Relative Motion with Applications to Autonomous Navigation and Control},
  author = {Willis, Matthew},
  school = {Stanford University},
  year   = {2023},
  url    = {https://purl.stanford.edu/kq677qp2544}
}

@article{koenig2017stm,
  title   = {New State Transition Matrices for Spacecraft Relative Motion in Perturbed Orbits},
  author  = {Koenig, Adam W. and Guffanti, Tommaso and D'Amico, Simone},
  journal = {Journal of Guidance, Control, and Dynamics},
  volume  = {40},
  number  = {7},
  pages   = {1749--1768},
  year    = {2017},
  doi     = {10.2514/1.G002409}
}

@article{guffanti2019srp,
  title   = {Linear Models for Spacecraft Relative Motion Perturbed by Solar Radiation Pressure},
  author  = {Guffanti, Tommaso and D'Amico, Simone},
  journal = {Journal of Guidance, Control, and Dynamics},
  volume  = {42},
  number  = {9},
  pages   = {1962--1981},
  year    = {2019},
  doi     = {10.2514/1.G002822}
}

@article{gaias2020analytical,
  title   = {Analytical Framework for Precise Relative Motion in Low {Earth} Orbits},
  author  = {Gaias, Gabriella and Colombo, Camilla and Lara, Martin},
  journal = {Journal of Guidance, Control, and Dynamics},
  volume  = {43},
  number  = {5},
  pages   = {915--927},
  year    = {2020},
  doi     = {10.2514/1.G004716}
}

@article{brouwer1959,
  title   = {Solution of the Problem of Artificial Satellite Theory Without Drag},
  author  = {Brouwer, Dirk},
  journal = {The Astronomical Journal},
  volume  = {64},
  year    = {1959},
  doi     = {10.1086/107958}
}

@article{chernick2018,
  title   = {New Closed-Form Solutions for Optimal Impulsive Control of Spacecraft Relative Motion},
  author  = {Chernick, Michelle and D'Amico, Simone},
  journal = {Journal of Guidance, Control, and Dynamics},
  volume  = {41},
  number  = {2},
  pages   = {301--319},
  year    = {2018},
  doi     = {10.2514/1.G002848}
}

@article{chernick2021,
  title   = {Closed-Form Optimal Impulsive Control of Spacecraft Formations Using Reachable Set Theory},
  author  = {Chernick, Michelle and D'Amico, Simone},
  journal = {Journal of Guidance, Control, and Dynamics},
  volume  = {44},
  number  = {1},
  pages   = {25--44},
  year    = {2021},
  doi     = {10.2514/1.G005218}
}

@inproceedings{delurgio2024,
  title     = {Closed-Form Modeling and Control of Spacecraft Swarms in Eccentric Orbits},
  author    = {Delurgio, Nicholas and D'Amico, Simone},
  booktitle = {2024 IEEE Aerospace Conference},
  pages     = {1--20},
  address   = {Big Sky, MT},
  year      = {2024},
  doi       = {10.1109/AERO58975.2024.10521046}
}

@article{albert2023,
  title   = {Relative Motion in the Velocity Frame for Atmospheric Entry Trajectories},
  author  = {Albert, Samuel W. and Schaub, Hanspeter},
  journal = {Journal of Spacecraft and Rockets},
  volume  = {60},
  number  = {5},
  pages   = {1614--1624},
  year    = {2023},
  doi     = {10.2514/1.A35753}
}

@article{peters2021cotangential,
  title   = {Linear Cotangential Transfers and Safe Orbits for Elliptic Orbit Rendezvous},
  author  = {Peters, Thomas V. and Noomen, Ron},
  journal = {Journal of Guidance, Control, and Dynamics},
  volume  = {44},
  number  = {4},
  pages   = {732--748},
  year    = {2021},
  doi     = {10.2514/1.G005152}
}

@article{buratti2026fsm,
  title   = {A Finite State Machine Guidance Architecture for Autonomous Rendezvous with Arbitrarily Elliptic Targets},
  author  = {Buratti, Diego and Gaias, Gabriella and Torresan, Stefano and Peters, Thomas V. and Roque, Pedro},
  journal = {Aerospace},
  volume  = {13},
  number  = {3},
  pages   = {230},
  year    = {2026},
  doi     = {10.3390/aerospace13030230}
}

@misc{d2005relative,
  title       = {Relative Orbital Elements as Integration Constants of {Hill}'s Equations},
  author      = {D'Amico, Simone},
  howpublished = {{DLR-GSOC TN 05-08}, Oberpfaffenhofen, Germany},
  month       = dec,
  year        = {2005}
}

@article{schaub2001j2,
  title={J2 invariant relative orbits for spacecraft formations},
  author={Schaub, Hanspeter and Alfriend, Kyle T},
  journal={Celestial Mechanics and Dynamical Astronomy},
  volume={79},
  number={2},
  pages={77--95},
  year={2001},
  doi={10.1023/A:1011161811472},
  publisher={Springer}
}

@book{vallado2001fundamentals,
  title={{Fundamentals of Astrodynamics and Applications}},
  author={Vallado, David A},
  edition={4},
  year={2013},
  publisher={Microcosm Press}
}

\end{document}